\documentclass{article}
\usepackage{arxiv}

\usepackage{graphicx} % Required for 
\usepackage{amsmath,amsthm}
\usepackage{amsfonts}
\usepackage{amssymb}
\usepackage{mathrsfs}
\usepackage{xcolor}
\usepackage{comment}
\usepackage{amsmath}
\usepackage{hyperref}
\usepackage{cleveref}
\usepackage{subcaption} 
\usepackage{caption} 

\usepackage{algorithm}      % floating algorithm environment
\usepackage{algorithmicx}   % enhanced algorithmic
\usepackage{algpseudocode}

\newcommand{\bx}{\mathbf{x}}
\newcommand{\bz}{\mathbf{z}}

\newcommand{\bc}{\mathbf{c}}
\newcommand{\by}{\mathbf{y}}
\newcommand{\bu}{\mathbf{u}}
\newcommand{\bv}{\mathbf{v}}

\newcommand{\bdf}{\mathbf{f}}

\newcommand{\rd}{\mathrm{d}}

\newcommand{\bL}{\mathfrak{L}}

\newcommand{\mG}{\mathcal{G}}
\newcommand{\mF}{\mathcal{F}}
\newcommand{\mN}{\mathcal{N}}

\newcommand{\mR}{\mathcal{R}}

\newcommand{\mM}{\mathcal{M}}
\newcommand{\mT}{\mathcal{T}}
\newcommand{\mX}{\mathcal{X}}

\newcommand{\mL}{\mathcal{L}}
\newcommand{\bR}{\mathbb{R}}

\newcommand{\bB}{\mathbf{B}}
\newcommand{\bD}{\mathbf{D}}
\newcommand{\bdE}{\mathbf{E}}

\newtheorem{theorem}{Theorem}
\newtheorem{lemma}{Lemma}

\newtheorem{definition}{Definition}
\newtheorem{assumption}{Assumption}

\graphicspath{{figures/}}

\newcommand{\wx}[1]{\textcolor{blue}{\textbf{[Wuzhe: #1]}}}
\newcommand{\rj}[1]{\textcolor{orange}{\textbf{[RJ: #1]}}}

\title{Self-Supervised Amortized Neural Operators for Optimal Control: Scaling Laws and Applications}
\author{Wuzhe Xu\thanks{W. Xu (xu2224@purdue.edu) is with the department of mathematics, Purdue University.}, \qquad Jiequn Han\thanks{J. Han (jhan@flatironinstitute.org) is with the Center for Computational Mathematics, Flatiron Institute.}, \qquad Rongjie Lai\thanks{R. Lai (lairj@purdue.edu) is with the department of mathematics, Purdue University.} }
\date{May 2025}

\begin{document}

\maketitle
\begin{abstract}
Optimal control provides a principled framework for transforming dynamical system models into intelligent decision-making, yet classical computational approaches are often too expensive for real-time deployment in dynamic or uncertain environments. In this work, we propose a method based on self-supervised neural operators for \emph{open-loop} optimal control problems. It offers a new paradigm by directly approximating the mapping from system conditions to optimal control strategies, enabling instantaneous inference across diverse scenarios once trained. We further extend this framework to more complex settings, including dynamic or partially observed environments, by integrating the learned solution operator with Model Predictive Control (MPC). This yields a solution-operator learning method for \emph{closed-loop} control, in which the learned operator supplies rapid predictions that replace the potentially time-consuming optimization step in conventional MPC.  
This acceleration comes with a quantifiable \emph{price to pay}. Theoretically, we derive scaling laws that relate generalization error and sample/model complexity to the intrinsic dimension of the problem and the regularity of the optimal control function.
Numerically, case studies show efficient, accurate real-time performance in low–intrinsic-dimension regimes, while accuracy degrades as problem complexity increases.
Together, these results provide a balanced perspective: neural operators are a powerful novel tool for high-performance control when hidden low-dimensional structure can be exploited, yet they remain fundamentally constrained by the intrinsic dimensional complexity in more challenging settings. 
\end{abstract}

\section{Introduction}

Optimal control offers a unifying mathematical framework for transforming knowledge of dynamical systems into principled decision-making. It underlies some of the most consequential problems in science and engineering, with applications ranging from robotics \cite{massaro2023optimal, song2023reaching} and aerospace \cite{betts2010practical} to finance \cite{caldentey2006optimal, cornuejols2018optimization} and drug delivery \cite{chakrabarty2005optimal, tassa2014control}. Despite its central role, solving optimal control problems remains computationally demanding. Classical methods, such as open-loop approaches based on trajectory optimization or Pontryagin’s Maximum Principle \cite{betts2010practical,bryson2018applied}, closed-loop feedback schemes derived from the Hamilton–Jacobi–Bellman (HJB) equation \cite{bellman1957dynamic, bertsekas2012dynamic}, and direct methods \cite{garg2010unified, kelly2017introduction} that convert the problem into parametric optimization problems, as well as more recent deep learning approaches~\cite{han2016deep,nakamura2021adaptive,barry2025physics, mowlavi2023optimal}, are typically tailored to individual instances and must be recomputed when system parameters or environments change. This heavy computational burden poses a major obstacle for real-time control in dynamic and uncertain settings.

Neural operators \cite{li2020fourier,lu2021learning} offer a new opportunity to address this challenge by learning mappings between function spaces, enabling zero-shot inference on new instances without per-instance retraining. In optimal control problems, they have been used to build dynamics surrogates \cite{wang2021fast}, enforce optimality via adjoint or HJB formulations \cite{yong2024deep,lee2025hamilton}, and synthesize stabilizing controllers through backstepping gain
kernels \cite{bhan2023neural}. Some of these approaches rely on supervised training with ground-truth optimal policies, which requires expensive data generation from existing solvers, while others avoid this cost by introducing auxiliary surrogate dynamics or reformulated optimality systems. In contrast, our approach avoids both ground-truth datasets and surrogate reformulations: we learn the mapping from environment functions directly to the optimal open-loop control in the original problem formulation, using self-supervised training to minimize the cost functional at the model’s predictions~\cite{kunisch2023learning,zhao2024offline}. This direct formulation can also be viewed as a form of amortized optimization~\cite{amos2023tutorial}, where a model learns to predict solutions across families of related optimization problems.

% \wx{Moreover, by embedding the learned neural operator into a receding-horizon MPC scheme, we obtain an approximate closed-loop controller that inherits the zero-shot generalization and real-time efficiency of the operator while enabling feedback adaptation to online changes in the environment.}

% This reusability enables high-performance, real-time deployment in domains where responsiveness and adaptability are essential. In practice, however, this efficiency collides with a data bottleneck: supervised training requires ground-truth optimal policies for many instances, each of which must be obtained individually using classical solvers, making dataset generation computationally expensive. To sidestep this bottleneck, we adopt self-supervised (physics-informed) training \cite{li2024physics,wang2021learning}, which minimizes residuals of the governing ODE/PDEs (and boundary/terminal conditions) at the model’s predictions, enabling learning directly from the formulation of the problem. 

Once trained, neural operators deliver efficient inference of \emph{open-loop} optimal control across different problem instances for high-performance real-time deployment. However, the prediction accuracy, characterized by generalization error, still depends on the problem’s intrinsic dimension and regularity, and this dependence is fundamental. A broad line of work shows that many high-dimensional applications admit hidden low-intrinsic-dimensional structure  \cite{osher2017low,djuric2015hate,tenenbaum2000global} and neural operators can exploit such structure \cite{chen2023deep,kovachki2021universal,lanthaler2022error} and therefore perform well in these regimes. When such low-intrinsic-dimensional structure is limited or absent, however, neural operator learning encounters the same barriers as classical numerical methods. In our setting, the \emph{price to pay} for learning a neural operator is its dependence on the intrinsic dimensionality of the control problem, reflecting a new form of Bellman’s curse of dimensionality~\cite{bellman1957dynamic}. To quantify this fundamental cost of learning, we provide both theoretical and practical insights. On the theoretical side, we establish scaling laws that quantify the sample and model complexity required to approximate the optimal control operators under varying intrinsic dimensions. On the practical side, we explore numerical case studies that highlight both the remarkable efficiency neural operators can deliver in structured regimes, and the limitations that persist in more complex environments. Taken together, our results offer a balanced perspective: neural operators represent a powerful and versatile tool for optimal control, capable of transforming how control is computed and deployed, but their benefits come with structural requirements that fundamentally shape their scope of applicability. By clarifying both the potential and the limits, we aim to shift the discussion from promise alone to a grounded understanding of what neural operators can and cannot achieve for optimal control, and more generally in other applications of operator learning.

% Yet, there is no free lunch.

% While recent efforts demonstrate encouraging results in low-dimensional or structured settings, major challenges remain in scaling these methods to high-dimensional tasks and generalizing beyond training distributions. Neural operators can exploit hidden low-dimensional structures—such as smooth dependence of optimal strategies on system parameters—but when such structures are limited or absent, the learning problem faces the same barriers as classical numerical methods. 

% In other words, the \emph{price to pay} for neural operator learning is a dependence on the intrinsic dimensionality and regularity of the control function, reflecting a new manifestation of Bellman’s curse of dimensionality~\cite{bellman1957dynamic}.

% In this work, we provide both theoretical and practical insights into neural operator learning for optimal control. On the theoretical side, we establish scaling laws that quantify the sample and model complexity required to approximate control operators under varying intrinsic dimensions. On the practical side, we explore numerical case studies that highlight both the remarkable efficiency neural operators can deliver in structured regimes, and the limitations that persist in more complex environments. Taken together, our results offer a balanced perspective: neural operators represent a powerful and versatile tool for optimal control, capable of transforming how control is computed and deployed, but their benefits come with structural requirements that fundamentally shape their scope of applicability.

\subsection{Problem Setup} 
We start with an \emph{open-loop} optimal control problem
\begin{equation}\label{eqn:Loss}
     \bu^*(t; B, \bx_0) \in  \arg \min_{\bu(t)} J_B(\bx(t), \mathbf{u}(t)) = \int_0^T L_B(t, \bx(t), \bu(t)) \rd t +  M(\bx(T))
\end{equation}
subject to the system dynamics
\begin{equation}\label{eqn:dynamics}
\left\{
\begin{aligned}
    \dot{\mathbf{x}}(t) &= \bdf(t, \mathbf{x}(t), \mathbf{u}(t)), \quad t \in [0, T] \\
    \mathbf{x}(0) &= \mathbf{x}_0 ,
\end{aligned}
\right.
\end{equation}
where $\bx_0 \in \Omega := \mathbb{R}^D$ is the initial state, $\bu(t) \in \mathcal{S} = \mathbb{R}^q$ is the control function lies in the admissible set $\mathcal{S}$, and $\bdf$ describes the dynamics.
$L_{B}(t, \bx(t), \bu(t))$ and $M(\bx(t))$ denotes the running cost and terminal cost, respectively. To model variations across problem instances, we introduce an \emph{environment function} $B \in C([0, T] \times \Omega)$ that characterizes the external conditions or task-specific factors of the control problem. For notational simplicity, we let $B$ enter only through the running cost $L_B$; allowing $B$ to also affect the terminal cost or dynamics introduces no conceptual difficulty and can be handled analogously.

% \wx{We note that the proposed learning framework is highly flexible and readily extends to various settings. A particularly important extension is to combine learned operator with Model Predictive Control (MPC) to enable closed-loop control predictions, which will be discussed in Section~\ref{sec:extension}. For clarity of presentation, the main body of this paper focuses on the \emph{open-loop} optimal control setting. Unless otherwise stated, throughout this paper we use the term “optimal control” to refer specifically to the \emph{open-loop} setting.} 

As the main body of the paper focuses on the above \emph{open-loop} optimal control problem, unless otherwise stated, we use the term “optimal control” to refer specifically to this \emph{open-loop} formulation, which encompasses a broad class of real-world problems, ranging from robotics and autonomous driving to energy systems, biomedical applications, and finance. For instance, in navigation tasks, the environment function $B$ encodes obstacle configurations, and the running cost often combines an energy term with an environment-dependent penalty,
\[
L_B(t, \bx(t), \bu(t)) = \ell(\bx(t),\bu(t)) +  B(t, \bx(t)).
\]
where $\ell(\bx, \bu)$ typically represents an energy function, such as kinetic energy. In this paper, we illustrate such formulations through three representative examples (see Section~\ref{sec:numerics}). 
Similar formulations appear in optimal drug dosing~\cite{tassa2014control}, where $B$ encodes patient-specific factors such as metabolism rates, sensitivity thresholds, or allowable toxicity levels; in energy systems~\cite{salas2018benchmarking}, where $B$ reflects time-varying demand or renewable availability; and in finance~\cite{ cornuejols2018optimization,han2021recurrent}, where $B$ captures market conditions. 
% Another example arises in optimal drug dosing \cite{tassa2014control}, where the running cost combines tumor burden with a penalty term that limits treatment toxicity. This toxicity penalty is directly tied to the environment function $B$, which encodes patient-specific factors such as metabolism rates, sensitivity thresholds, and allowable toxicity levels.

Consider the solution operator $\mG$ that maps a problem instance tuple $(B, \bx_0)$ to the corresponding optimal control function:
\begin{eqnarray}
\label{eqn:solutionoperator}
  \mG : C([0,T]\times\Omega) \times \Omega 
    \; &\longrightarrow\; C([0,T],\mathcal{S}) \nonumber \\
    (B, \bx_0) & \mapsto \bu^*(t; B, \bx_0)
\end{eqnarray}
The major goal of this paper is to parametrize $\mG$ by a neural operator $\mG_\theta$. A general bottleneck for neural operator approach is the scarcity of reliable training data. Here we assume that the cost functional $J_B$ and dynamics function $\bdf$ are known, and we develop a self-supervised, amortized training method. Specifically, given a neural operator $\mG_\theta$, its output (the predicted control function) can be propagated through the dynamics using a Neural ODE solver to generate the corresponding trajectory. By evaluating this predicted trajectory within the cost functional and minimizing the resulting objective, the model can then be trained in self-supervised manner, which will be detailed in Section~\ref{sec:method}. We then develop a theoretical understanding of the neural scaling laws associated with the proposed operator-learning method, quantifying the sample and model complexity required to approximate optimal control operators across varying intrinsic dimensions. We note that the proposed learning framework is highly flexible and readily extends to various settings. A particularly important extension is to combine learned operator with Model Predictive Control (MPC) to enable closed-loop control predictions, which will be discussed in Section~\ref{sec:extension}. 

While many deep reinforcement learning formulations treat the environment condition $B$ as part of the state variable and aim to learn a single policy that generalizes across different $B$ \cite{hessel2019multi, zhou2019environment, lee2020context}, our setting regards $B$ as an external, fixed condition that defines a particular control problem instance. This distinction is more natural from a control perspective: $B$ characterizes immutable environmental factors that remain unchanged under control, rather than dynamic states to be influenced. Moreover, this treatment enables a clearer theoretical analysis: by conditioning on $B$, we can investigate the sample complexity and generalization behavior of the learned operator $\mG_\theta$ across problem instances, providing quantitative insights into its ability to adapt to new environments.

%We summarize our contributions below: 
\paragraph{Major Contributions.} 
This work makes the following contributions to the study of neural operator learning for optimal control:

\begin{itemize}
    \item \emph{Neural operator formulation of optimal control.} We formulate optimal control problems in the amortized operator learning setting, where the mapping from system conditions to optimal strategies is learned directly.  The framework leverages system dynamics and the cost functional as self-supervision, enabling learning of optimal controls across problem instances without requiring ground-truth data. 
    
    \item \emph{Theoretical scaling laws.} We establish sample complexity and approximation guarantees that quantify how the feasibility of neural operator learning depends on the intrinsic dimension and regularity of the control-to-solution map.  
    
    \item \emph{Computational trade-offs and structural limits.} We characterize the ``price to pay'' for neural operator learning, identifying regimes where hidden low-dimensional structures enable efficiency and where dimensional complexity dominates. These results highlight both the promise and the limitations of neural operators, clarifying when they offer transformative gains and when classical barriers remain.  

    \item \emph{Extension to closed loop optimal control.}  
    We integrate the learned solution operator into a model predictive control framework, yielding an efficient closed-loop controller that adapts to online changes in the environment.
     
    \item \emph{Numerical demonstrations.} We provide case studies that showcase the advantages of neural operators for real-time, high-performance control, while also illustrating their limitations in high-dimensional or weakly structured problems.    
\end{itemize}

\subsection{Structure of the Paper} Section~\ref{sec:method} introduces the proposed framework and its amortized training procedure, together with a theoretical guarantee of optimality. In Section~\ref{sec:thm}, we present the main scaling-law result, which highlights the dependence of neural operator learning on the intrinsic dimension and formalizes our central argument regarding the curse of dimensionality, and Section~\ref{sec:proof} provides the detailed proof. Section~\ref{sec:extension} extends the methodology to Model Predictive Control (MPC), enabling the treatment of optimal control problems that require real-time responsiveness. Finally, Section~\ref{sec:numerics} demonstrates the effectiveness of the proposed method through a series of experiments and offers numerical validation of the predicted scaling behavior.

\section{Methodology}\label{sec:method}
In this section, we first present the amortized training framework for obtaining the neural operator, followed by a theoretical guarantee of its optimality. After that, we discuss our training methodology. All notations are provided  Table~\ref{tab:notation} in the appendix.

\subsection{Amortized Training}
Let $\Omega$ denote the state space. Let $\rho \in \mathcal{P}(\Omega)$ be a probability measure on $\Omega$ for the initial state, and $\mu \in \mathcal{P}(C([0,T]\times\Omega))$ be a probability measure on the space of continuous functions $C([0,T]\times\Omega)$ for the environmental function. In many applications (e.g., \cite{osher2017low,djuric2015hate,tenenbaum2000global}), the intrinsic dimension of the data is much smaller than the ambient dimension. Adapting this perspective to our setting means that the supports of $\rho$ and $\mu$ are effectively low-dimensional relative to $\Omega$ and $C([0,T]\times\Omega)$, respectively. For example, if the environment function $B$ is an obstacle function that can be well approximated by a Gaussian mixture, then $B$ need not be treated as an infinite-dimensional function; instead it can be parameterized by a finite set of mixture parameters (means, covariance matrices, and component weights), yielding a low intrinsic dimension. Consequently, throughout this paper we assume that $\rho$ is supported on a $d$-dimensional Riemannian submanifold $\mathcal X \subset \Omega$, and that $\mu$ is supported on a $k$-dimensional Riemannian submanifold $\mathcal M \subset C([0,T]\times\Omega)$ (see Assumption~\ref{assump:low_d} for the precise statement). We then consider the following amortized training: 
\begin{eqnarray} \label{eqn:amortize_poploss}
    \mG^* &=&  \arg \min_{\mG} \mR(\mG) = \arg \min_{\mG}  \mathbb{E}_{B\sim \mu}\mathbb{E}_{\bx_0\sim\rho}  J_B(\mathbf{x}(t), \mG(B,\bx_0)) \nonumber \\
   &&\text{s.t.} \left\{
\begin{aligned}
    \dot{\mathbf{x}}(t) &= \bdf(t, \mathbf{x}(t); \mG(B,\bx_0)(t)), \quad t \in [0, T] \\
    \mathbf{x}(0) &= \mathbf{x}_0
\end{aligned}
\right.
\end{eqnarray}
This optimization problem is the central focus of our work. Our first result, Theorem~\ref{thm:optimality}, shows that the minimizer of the empirical version \eqref{eqn:amortize_emploss_nn} coincides almost everywhere (with respect to $\mu\times\rho$) with the true solution operator.
\begin{theorem}\label{thm:optimality}
   If $\mG^*$ is obtained by solving the amortized training problem \eqref{eqn:amortize_poploss}, then $\mG^*$ is a solution operator up to measure zero sets with respect to measure $\nu = \rho \times \mu$.  That is, for $\nu-$almost every pair $(B, \bx_0)$,
   \[
   \mG^*(B, \bx_0)(t) = \bu^*(t) \in \arg \min_{\bu(t)} J_B(\bx(t), \bu(t)), 
   \]
   with $\bx(t)$ satisfying \eqref{eqn:dynamics}. 
   % Then, for almost every $\bx_0 \sim \rho$ and $B\sim \mu$, $\mG(B, \bx_0) = \arg\min_{\bu} \mathbb{E}_{B\sim \mu} \mathbb{E}_{\bx_0\sim\rho} J_B(\bx(t),\bu)$ with $\bx(t)$ satisfying \eqref{eqn:dynamics}. 
\end{theorem}

The proof of \cref{thm:optimality} directly follows the idea of Theorem 4.3 in \cite{huang2024unsupervised}.
\begin{proof}
Let $S = \{ (B,\bx_0) | \mG^*(B, \bx_0)>\arg \min_{\bu} J_B(\bx(t), \bu) \text{ and $\bx(t)$ satisfying \eqref{eqn:dynamics}} \}$ and denote its complement $S^c =(C([0,T]\times\Omega) \times  \Omega) \setminus S$. Consider another operator $\hat{\mG}$ such that $\hat{\mG}|_{S^c} = \mG^*|_{S^c}$ and $\hat{\mG}(B, \bx_0) \in \arg \min_{\bu} J_B(\bx(t), \bu)$ with $\bx(t)$ satisfying \eqref{eqn:dynamics}, $\forall (B, \bx_0) \in S$. Now suppose $\nu(S)>0$, which implies
\begin{align*}
    \mathbb{E}_{B\sim \mu}\mathbb{E}_{\bx_0\sim\rho}  J_B(\mathbf{x}(t), \mG^*(B,\bx_0)) &= \int_{S} J_B(\mathbf{x}(t), \mG^*(B,\bx_0)) \rd \nu(B, \bx_0) + \int_{S^c} J_B(\mathbf{x}(t), \mG^*(B,\bx_0)) \rd \nu(B, \bx_0), \\
    & > \int_{S} J_B(\mathbf{x}(t), \hat{\mG} (B,\bx_0)) \rd \nu(B, \bx_0) + \int_{S^c} J_B(\mathbf{x}(t), \hat{\mG} (B,\bx_0)) \rd \nu(B, \bx_0), \\
    & = \mathbb{E}_{B\sim \mu}\mathbb{E}_{\bx_0\sim\rho}  J_B(\mathbf{x}(t), \hat{\mG}(B,\bx_0)),
\end{align*}
%where the inequality follows the Lemma 6.1 in \cite{huang2024unsupervised}. 
This result implies that $\hat{\mG}$ achieves strictly smaller population risk than $\mG^*$, contradicting the optimality of $\mG^*$ for the amortized training problem \eqref{eqn:amortize_poploss}. Thus $\nu(S) = 0$.
\end{proof}

\subsection{Training Algorithm} 
We generate a training set by drawing i.i.d. tasks $(B^i,\bx_0^i)\sim\mu\times\rho$ for $i=1,\dots,n$, yielding $\Gamma=\{(B^i,\bx_0^i)\}_{i=1}^n$. Our goal is to approximate the target solution operator $\mG^*$ using a neural network $\mG_\theta$ drawn from a hypothesis class $\mathscr{G}$. For the purposes of analysis, we focus on deep ReLU networks. We emphasize, however, that the proposed framework is compatible with a wide range of network architectures in practical implementations.

We specify deep ReLU networks as follows. 
For a given input $\mathbf{z} \in \mathbb{R}^{n_0}$, an $l$-layer,   ReLU neural network is defined as follows
\begin{equation}\label{eqn:relu_def}
    g^l_\theta(\bz) = W_l \cdot \operatorname{ReLU}\bigl( W_{l-1} \cdots \operatorname{ReLU}\bigl( W_1 \bz + \mathbf{b}_1 \bigr) \cdots + \mathbf{b}_{l-1} \bigr) + \mathbf{b}_l.
\end{equation}
For each hidden layer $i = 1, 2, \ldots, l-1$, the weight matrix and bias vector are denoted by  
$W_i \in \mathbb{R}^{n_i \times n_{i-1}}$ and $\mathbf{b}_i \in \mathbb{R}^{n_i}$, respectively, where $n_i$ represents the width of layer $i$. The overall width of the network is defined as  
\begin{equation}
    p = \max_{i \in \{1, 2, \ldots, l-1\}} n_i,
\end{equation}
and all trainable parameters are collected in  
$\theta = \{ W_1, \mathbf{b}_1, W_2, \mathbf{b}_2, \ldots, W_l, \mathbf{b}_l \}$. 

Because the cost functional and dynamics depend on trainable parameters $\theta$ only through the neural operator $\mG_\theta$, we will use two equivalent forms throughout: the parameter form for optimization and the operator form for theoretical analysis. Specifically, we define following empirical loss function on training set $\Gamma$:
\begin{equation}\label{eqn:amortize_emploss_nn}
 \mL(\theta;\Gamma) =  \mR_\Gamma(\mG_\theta) =  \frac{1}{n}\sum_{i = 1}^n 
 J_{B^i}\!\left(\mathbf{x}_\theta^i(t), \mG_\theta(B^i,\bx^i_0)\right)
\end{equation}
subject to
\begin{equation}\label{eqn:amortize_dynamics_nn}
\left\{
\begin{aligned}
    \dot{\mathbf{x}}_\theta^i(t) &= \bdf\!\left(t, \mathbf{x}_\theta^i(t); \mG_\theta(B^i,\bx^i_0)(t)\right), \quad t \in [0, T], \\
    \mathbf{x}_\theta^i(0) &= \mathbf{x}^i_0.
\end{aligned}
\right.
\end{equation}
Here the dynamics are modeled using a neural ODE \cite{chen2018neural} and $\mathbf{x}_\theta^i(t)$ denotes the trajectory obtained by solving the dynamics~\ref{eqn:amortize_dynamics_nn} with the control generated by the network $\mG_\theta$. In the parameter form, the optimal network parameters are obtained by minimizing the empirical loss:
\begin{equation}\label{eqn:emp_loss_theta}
    \theta^* = \arg\min_\theta \mL(\theta; \Gamma),
\end{equation}
Equivalently, in operator form, the empirical risk minimizer within the deep ReLU network class $\mathscr{G}$ for the training set $\Gamma$ is given by
\begin{equation}\label{eqn:emp_loss_G}
\mG_{\Gamma;\theta^*} = \arg\min_{\mG_\theta \in \mathscr{G}} \mR_\Gamma(\mG_\theta).
\end{equation}
The overall training procedure is summarized in Algorithm~\ref{alg:training}.
\begin{algorithm}[htbp]
    \caption{Self-Supervised Amortized Training}
    \label{alg:training}
    \begin{algorithmic}[1]
        \Require Training dataset $\Gamma = \{(B^i,\bx_0^i)\}_{i=1}^n \in \mM\times \mX$, number of epochs $N$, batch size $B_s$.
        \State Initialize network parameters $\theta$
        \For{epoch = 1 to $N$} 
        % \jh{To fix on enumerating datasets}
        \State Shuffle $\Gamma$
        \For{each mini-batch $\{(B^j,\bx_0^j)\}_{j=1}^{B_s}$ from $\Gamma$}
        % \State Sample mini-batch $\{(B^j,\bx_0^j)\}_{j=1}^{B_s} \sim \Gamma$ 
        \State Simulate trajectories by solving the neural ODE with controls generated by $\mG_\theta$
        \[
        \dot{\bx}_\theta^j(t) = \bdf\!\big(t,\bx_\theta^j(t);\mG_\theta(B^j,\bx^j_0)(t)\big),
        \quad \bx_\theta^j(0) = \bx^j_0, \quad j=1, \cdots, B_s.
        \]
        %\State Update $\theta$ using Adam to minimize the empirical loss \eqref{eqn:amortize_emploss_nn} \rj{I assume you also update the adjoin state using Neural ODE package. This reads like it detached ODE constraints. }

      % \State Compute the empirical loss $\mathcal{L}(\theta)$ and its gradient $\nabla_\theta \mathcal{L}$ by backpropagating through the Neural ODE solver\cite{chen2018neural} using the adjoint method implemented in PyTorch~\cite{paszke2019pytorch}, and then update $\theta$ using Adam

      \State Compute the empirical loss $\mathcal{L}(\theta)$ on the mini-batch and obtain $\nabla_\theta \mathcal{L}(\theta)$ by backpropagating through the ODE solver in PyTorch~\cite{paszke2019pytorch}, and then update $\theta$ using Adam.
        
        \EndFor
        \EndFor
        \State \Return Trained model $\mG_{\Gamma;\tilde{\theta}^*}$
    \end{algorithmic}
\end{algorithm}

We remark that in the above algorithm, we assume a fixed dataset $\Gamma$ and train for $N$ epochs. When the computational cost of sampling $(B, \bx_0)$ from $\mM \times \mX$ is low, one can instead generate fresh samples in each batch. This is effectively equivalent to using a dataset of size $nN$, with each sample seen only once. Such a procedure may improve performance by reducing generalization error, but it only increases the total sample size by a constant factor and does not affect the exponential complexity presented below. To avoid unnecessary complications in analysis, we present the algorithm using a fixed dataset. Additionally, if the state $\bx$ lives in high dimension and the neural ODE requires a small time step for stability and accuracy, one can instead employ the adjoint method for Neural ODEs~\cite{chen2018neural} to compute gradients while substantially reducing memory usage (at the expense of additional computation from solving an associated adjoint ODE backward in time).

\section{Scaling Law of Generalization Error}\label{sec:thm}

This section dedicates to analyzing the generalization error of the solution operator $\mG_{\Gamma;\theta^*}$ obtained from \eqref{eqn:emp_loss_G} given a training dataset $\Gamma = \{(B^i,\bx_0^i)\}_{i=1}^n$.
More precisely, we study the error bound for the generalization error $\mathbb{E}_\Gamma [\mR(\mG_{\Gamma;\theta^*}) - \mR(\mG^*) ]$, where the excess risk of the empirical minimizer:
\begin{equation}
\label{eqn:genearlizationerror}
\mR(\mG^*) - \mR(\mG_{\Gamma;\theta^*}) =
\mathbb{E}_{B\sim \mu} \mathbb{E}_{\bx_0\sim\rho}  (J_B(\mathbf{x}(t), \mG^*(B,\bx_0)) - J_B(\mathbf{x}(t), \mG_{\Gamma;\theta^*}(B,\bx_0)))
\end{equation}
quantifies the performance gap on unseen tasks with $\bx_0 \sim \rho$ and $B \sim \mu$. Note that, in practice, the trained operator $\mG_{\Gamma;\tilde{\theta}^*}$ obtained from Algorithm~\ref{alg:training} generally differs from the empirical risk minimizer $\mG_{\Gamma;\theta^*}$ due to optimization error. In the following generalization error analysis, we focus on the latter, $\mG_{\Gamma;\theta^*}$, a global minimizer from the empirical training loss. 
It is worth noting that the gap between the practically obtained operator 
$\mG_{\Gamma;\tilde{\theta}}$ and the empirical loss minimizer 
$\mG_{\Gamma;\theta^*}$ may, in principle, be controlled under additional 
assumptions on the optimizer and the training dynamics. However, 
characterizing this error rigorously is substantially more challenging and 
remains an open problem, lying beyond the scope of the present work.

In general, the solution operator \eqref{eqn:solutionoperator} of the optimal control problem is a map
acting between infinite-dimensional function spaces.  
From a learning perspective, directly approximating such an operator is challenging unless additional structure is present.  
Recent theoretical analyses of deep neural networks and neural operators demonstrate that sample and model complexity typically scale with the \emph{intrinsic dimension} of the input space, rather than its ambient dimension.  
Motivated by this observation, we assume that the distributions of the inputs are supported on low-dimensional manifolds:  
the initial-state distribution $\rho$ is supported on a $d$-dimensional Riemannian submanifold $\mathcal{X}\subset \Omega$, and the distribution $\mu$ of the environment field $B$ is supported on a $k$-dimensional Riemannian submanifold $\mathcal{M}\subset C([0,T]\times\Omega)$.  
This assumption is natural in many applications.  
The region of interest for initial conditions often lies in a lower-dimensional subset of the ambient state space (for example, physically realizable configurations or states constrained by conservation laws), and the environment $B$ is typically governed by only a few parameters---such as temperature profiles, material coefficients, boundary conditions, or forcing amplitudes---despite being represented as a high-dimensional function.  
Consequently, although the operator $\mathcal{G}$ is formally defined on infinite-dimensional domains, the \emph{effective} domain encountered in practice is finite-dimensional, making operator learning both meaningful and tractable. This leads to the following low-dimension manifold assumption.

\begin{assumption}[Low-dimensional bounded manifold]\label{assump:low_d}
    We consider that the initial-state distribution $\rho$ is supported on a $d$-dimension Riemannian manifold $\mX \subset \Omega=\mathbb{R}^{D}$ and environment distribution $\mu$ is supported on a $k$-dimension Riemannian manifold $\mM \subset C([0,T] \times \Omega)$. Both $\mX$ and $\mM$ have reach $\tau>0$. Moreover, there exists a constant $E>0$ such that, for any $(B, \bx_0) \in \mM\times \mX$, we have 
    \[
    |(\bx_0)_j|\le E \quad \text{for all } j=1,\dots,D,
    \qquad\text{and}\qquad
    |B(\bx)|\le E \quad \text{for all } \bx\in\Omega.
    \]
\end{assumption}

Our goal is to characterize how the generalization error \ref{eqn:genearlizationerror} and network complexity depend on the intrinsic dimensions of the manifolds $\mX$ and $\mM$, thereby revealing the impact of the curse of dimensionality. We begin by stating assumptions.

% \begin{assumption}[Lipschitz continuity]\label{assump:Lip_cont} 
% We assume $J, L$ and $M$ defined  in \eqref{eqn:Loss} and $f$ defined in \eqref{eqn:dynamics} are Lipschitz continuous. More precisely, we have: \rj{why we need to assume both $J$ and $L, M$?}
%    \begin{eqnarray}
%    \sup_{B,\bx_0}|J_B(\bx(t),\bu)  - J_B(\bx'(t),\bu') | &\leq& \bL^J  \|\bu - \bu'\|_{\infty},
%         \\ 
%     |L(t,\bx,\bu,B)  - L(t',\bx',\bu', B')| &\leq& \bL^L_t |t - t'| + \bL^L_\bx \|\bx - \bx'\|_{\infty} + \bL^L_\bu \|\bu - \bu'\|_{\infty}, \\
%   |M(\bx)  - M(\bx')| &\leq& \bL^M \|\bx - \bx'\|_{\infty},  \\
%        |\bdf(t, \mathbf{x}; \mathbf{u}) - \bdf(t', \mathbf{x}'; \mathbf{u}')| &\leq& \bL^\bdf_t |t - t'| + \bL^\bdf_\mathbf{x}\|\mathbf{x} - \mathbf{x}'\|_{\infty} + \bL^\bdf_{\mathbf{u}}\|\mathbf{u} -\mathbf{u}'\|_{\infty}.
%    \end{eqnarray}
% \end{assumption}

\begin{assumption}[Lipschitz continuity]\label{assump:Lip_cont} 
We assume $L$ and $M$ defined in \eqref{eqn:Loss} and $f$ defined in \eqref{eqn:dynamics} are Lipschitz continuous. More precisely, we have:
   \begin{eqnarray}
    |L(t,\bx,\bu,B)  - L(t',\bx',\bu', B')| &\leq& \bL^L_t |t - t'| + \bL^L_\bx \|\bx - \bx'\|_{\infty} + \bL^L_\bu \|\bu - \bu'\|_{\infty}, \\
  |M(\bx)  - M(\bx')| &\leq& \bL^M \|\bx - \bx'\|_{\infty},  \\
       |\bdf(t, \mathbf{x}; \mathbf{u}) - \bdf(t', \mathbf{x}'; \mathbf{u}')| &\leq& \bL^\bdf_t |t - t'| + \bL^\bdf_\mathbf{x}\|\mathbf{x} - \mathbf{x}'\|_{\infty} + \bL^\bdf_{\mathbf{u}}\|\mathbf{u} -\mathbf{u}'\|_{\infty}.
   \end{eqnarray}

\end{assumption}

These Lipschitz continuity assumptions are standard and are satisfied in many practical settings, for example, in the numerical experiments we present later on obstacle avoidance, maze navigation, and the unicycle model. As a result, there exist a $\bL^J>0$,
depending only on the
Lipschitz constants $\bL^L_\bx, \bL^L_\bu, \bL_M, \bL^\bdf_\bx, \bL^\bdf_\bu$ and $T$ in \eqref{eqn:Loss}, such that
\[
\sup_{B,\bx_0}|J_B(\bx(t),\bu)  - J_B(\bx'(t),\bu') | \leq \bL^J  \|\bu - \bu'\|_{\infty}
\]
We further assume a regularity condition on the control, stated below.
\begin{assumption}[Regularity of control]
    For all $\bx_0 \in \mX$ and $B \in \mM$, we assume the optimal control function $\bu^*(t; B, \bx_0) \in \mathcal{H}^{s,\alpha}(\mathbb{R}^+ \times \mX \times \mM)$, with positive integer $s$ and $\alpha>0$.
\end{assumption}

\begin{assumption}[Uniform boundedness of the cost functional]\label{assump:bound_cost}
We assume that for all $(B, \bx_0)\in \mM\times \mX$ and all parameters $\theta$, let $\bx(t)$ be the trajectory driven $\mG_\theta$ according to the dynamics described in \eqref{eqn:dynamics},  $|J_B(\bx(t),\mathcal{G}_\theta(B,\bx_0))| \leq R^J$  holds uniformly.
    
\end{assumption}

In addition, we specify below the class of neural networks that defines the hypothesis space over which the empirical loss is optimized.
\begin{definition} [Bounded ReLU Network Class]
We denote $\mathscr{F}(R,\kappa,l,p,K)$ as a class of ReLU networks expressed as \eqref{eqn:relu_def} with depth $l$, width bounded by $p$, bounded weight parameters and bounded output, that is for $g \in \mathscr{F}(R,\kappa,l,p,K)$, $f(\bz)$ satisfies $(\bz = (t, B, \bx_0))$
\[
\|g\|_{\infty} \le R,\ \|W_i\|_{\infty,\infty} \le \kappa,\ \|\mathbf{b}_i\|_{\infty} \le \kappa \ \text{for } i=1,\ldots,l, \ 
\sum_{i=1}^l \|W_i\|_{0} + \|\mathbf{b}_i\|_{0} \le K.
\]
\end{definition}

\begin{definition}
[Bounded Vector-Valued ReLU Network Class]
For $q \in \mathbb{N}$, we denote by $\mathscr{G}(q, R,\kappa,l,p,K)$ the class of 
$q$-dimensional ReLU networks whose components belong to $\mathscr{F}(R,\kappa,l,p,K)$.  
That is, for $\mG \in \mathscr{G}(q, R,\kappa,l,p,K)$, 
\[
\mG(\bz) = \bigl(g_1(\bz),\ldots,g_q(\bz)\bigr)^\top,\qquad 
\bz = (t,B,\bx_0),
\]
where each $g_i \in \mathscr{F}(R,\kappa,l,p,K)$ for $i=1,\ldots,q$.
\end{definition}
Under these assumptions, we can now establish the following generalization error bound. 
\begin{theorem}[Generalization Error]\label{thm:gen}
Given a training set $\Gamma = \{(B^i,\bx_0^i)\}_{i=1}^n \in \mM \times \mX$, there exists a neural operator $\mG_{\Gamma;\theta^*}$ in ReLU network class $\mathscr{G}(q, R, \kappa, l, p, K)$ with $l = \widetilde{O}(\frac{s+\alpha}{d+k+2(s+\alpha)} \log n)$, $p=\widetilde{O}(n^{\frac{d+k}{2(\alpha+s)+d+k}})$, $K = \widetilde{O}\left(\frac{s+\alpha}{2(s+\alpha)+d+k} \, n^{\tfrac{d+k}{2(s+\alpha)+d+k}} \log n\right)$ and $\kappa = O(\max \{ 1, E, \sqrt{d+k}, \tau^2 \})$, such that
\[
\displaystyle \mathbb{E}_\Gamma [ \mR(\mG_{\Gamma;\theta^*}) - \mR(\mG^*)]  \leq C_1 \sqrt{n^{-\frac{2(s+\alpha)}{d+k+2(s+\alpha)}}+\frac{D+k}{n}}\log^{\frac{3}{2}}n + C_2 \sqrt{\frac{\log n}{n}} + C_3 n^{-\frac{s+\alpha}{d+k+2(s+\alpha)}},
\]
where $C_1$ is a constant depends on $q, R^J, E, \bL^J, \log D, d, k, \tau, \alpha$ and $s$; $C_2$ is a constant depends on $\alpha$ and $s$; and $C_3$ is a constant depends on $\bL^M, \bL^{\bdf},\bL^L,C_{\bdf}, R$ and $T$.
\end{theorem}

This theorem characterizes how the sample and model complexities scale with the intrinsic dimensions $d$ and $k$, together with the regularity parameters $\alpha$ and $s$ of the control function, when deep ReLU networks are used as the hypothesis class. A key insight is that although the ambient spaces in which the state and environment reside may be extremely high dimensional or even infinite dimensional, the relevant quantities to be inferred often lie on low dimensional manifolds. In particular, the initial state $\bx_{0}$ may vary only over a $d$ dimensional submanifold $\mathcal{X}\subset\mathbb{R}^{D}$, and the environment parameter $B$, though formally an element of a function space of infinite ambient dimension, may in practice be parameterized by a $k$ dimensional manifold $\mathcal{M}\subset C([0,T]\times\Omega)$.
This theorem makes this geometric structure precise: the sample and model complexities depend only on the intrinsic dimensions $d$ and $k$, and not on the ambient dimensions of the state space or function space. This provides formal evidence that neural networks can overcome the curse of ambient dimensionality.

At the same time, the result highlights an equally important phenomenon. The rate
$
n^{\frac{s+\alpha}{d+k+2(s+\alpha)}}
$
shows that the sample complexity deteriorates exponentially with the total intrinsic dimension $d+k$. This reflects an unavoidable curse of intrinsic dimensionality: as the manifolds $\mathcal{X}$ and $\mathcal{M}$ become higher dimensional, learning the optimal control operator necessarily becomes more difficult. This dependence is a fundamental cost of achieving computational efficiency in operator learning for optimal control.

The complete proof is provided in Section~\ref{sec:proof}.

%This theorem characterizes how the sample and model complexities scale with the intrinsic dimensions $d$ and $k$, as well as the regularity parameters $\alpha$ and $s$ of the control function, when deep ReLU networks are used as the hypothesis class. In particular, the term ${n^{-\frac{(s+\alpha)}{d+k+2(s+\alpha)}}}$, shows that the sample complexity grows exponentially with the total intrinsic dimension $d+k$. This highlights the curse of intrinsic dimensionality: as the manifolds $\mX$ and $\mM$ become higher-dimensional, the convergence rate deteriorates, which is an inevitable \emph{price to pay} for computational efficiency. The complete proof is presented in Section~\ref{sec:proof}.

\section{Proof of Main Theorem}\label{sec:proof}
This section provides detailed proof of the main theorem. 
%\subsection{Sketch of the Proof}
We begin with a sketch of the proof. Note that the excess risk of the empirical minimizer 
$\mG_{\Gamma;\theta^*}$ trained on a dataset $\Gamma$ can be decomposed as:
\begin{eqnarray}
        0 \leq \mR(\mG_{\Gamma;\theta^*}) - \mR(\mG^*)  & = &
    \mR(\mG_{\Gamma;\theta^*}) - \mR_\Gamma(\mG_{\Gamma;\theta^*}) +
    \mR_\Gamma(\mG_{\Gamma;\theta^*}) - \mR_\Gamma(\mG_{\hat\theta}) + 
    \mR_\Gamma(\mG_{\hat\theta})- \mR(\mG_{\hat\theta})   + 
     \mR(\mG_{\hat\theta}) - \mR(\mG^*) \nonumber \\
     &\leq &     \mR(\mG_{\Gamma;\theta^*}) - \mR_\Gamma(\mG_{\Gamma;\theta^*}) + 
    \mR_\Gamma(\mG_{\hat\theta})- \mR(\mG_{\hat\theta})   + 
     \mR(\mG_{\hat\theta}) - \mR(\mG^*) \label{ieq:negative}
\end{eqnarray}
where $\mG_{\hat\theta}$ is an approximator that we will construct on the hypothesis class and  the inequality \eqref{ieq:negative} follows the fact that
\[
\mR_\Gamma(\mG_{\Gamma;\theta^*}) = \min_{\mG_\theta\in\mathscr{G}} \mR_\Gamma(\mG) \leq  \mR_\Gamma(\mG_{\hat\theta}).
\]
Therefore, the generalization error can be decomposed into three parts:
\begin{eqnarray}\label{eqn:errordecomp0}
   \mathbb{E}_\Gamma (\mR(\mG_{\Gamma;\theta^*}) -\mR(\mG^*)) \leq 
    \underbrace{\mathbb{E}_\Gamma(\mR(\mG_{\Gamma;\theta^*}) - \mR_\Gamma(\mG_{\Gamma;\theta^*}))}_{{\rm I}} + 
    \underbrace{\mathbb{E}_\Gamma (\mR_\Gamma(\mG_{\hat\theta})- \mR(\mG_{\hat\theta}))}_{{\rm II}} + 
        \underbrace{(\mR(\mG_{\hat\theta}) - \mR(\mG^*))}_{{\rm III}}
\end{eqnarray}
where ${\rm I, II}$ correspond to the statistic errors and ${\rm III}$ characterizes approximation error. In the rest of this section, we derive bounds for each component and then combine them to complete the proof.

\subsection{Approximation Error Term {\rm III}}
First, we analyze the approximation error term~{\rm III} in the decomposition~\eqref{eqn:errordecomp0}, which reflects the inherent approximation capability of the hypothesis class. As a first step, we establish the following ODE stability result.
\begin{lemma}\label{lemma:ODE_controlPerturb}
Consider two ODEs with the same initial condition:
\begin{align*}
    \frac{d \bx}{dt} &= \bdf(t,\bx;\bu(t,\bx)), \quad \bx(0) = \bx_0, \\
    \frac{d \by}{dt} &= \bdf(t,\by;\bv(t,\by)), \quad \by(0) = \bx_0,
\end{align*}
where we assume that $ \|\bv(t,\bx) - \bu(t,\bx)\|_{\infty} \leq \epsilon $ uniformly and that $\bdf$ is Lipschitz continuous in $\bx$ and $\bu$  with constant $\bL^\bdf_\bx$ and $\bL^\bdf_\bu$, respectively.  Then
\[
\|\bx(t) - \by(t)\|_{\infty} \leq C_\bdf (e^{\bL^\bdf_\bx t} - 1) \epsilon.
\]
where $C_\bdf = \displaystyle \bL^\bdf_\bu / \bL^\bdf_\bx$. 
\end{lemma}

\begin{proof}
Let $ \delta(t) := \bx(t) - \by(t) $. Then,
\[
\frac{d\delta}{dt} = \frac{d\bx}{dt} - \frac{d\by}{dt} = \bdf(t,\bx;\bu(t,\bx)) - \bdf(t,\by;\bv(t,\by)).
\]
We add and subtract $ \bdf(t,\by;\bu(t,\bx))$:
\[
\frac{d\delta}{dt} = \bdf(t,\bx;\bu(t,\bx)) - \bdf(t,\by;\bu(t,\bx)) + \bdf(t,\by;\bu(t,\bx)) - \bdf(t,\by;\bv(t,\by)) 
\]
Taking norms, By the Lipschitz continuity of $ \bdf $,  we get:
\[
\left\| \frac{d\delta}{dt} \right\|_{\infty} \leq \\bL^\bdf_\bx \|\delta(t)\|_{\infty} + \bL^\bdf_\bu \epsilon.
\]
This gives a differential inequality:
\[
\frac{d}{dt} \|\delta(t)\|_{\infty} \leq \\bL^\bdf_\bx \|\delta(t)\|_{\infty} + \bL^\bdf_\bu \epsilon.
\]
Applying Grönwall's inequality:
\[
\displaystyle \|\bx(t) - \by(t)\|_{\infty} \leq \bL^\bdf_\bu \epsilon \int_0^t e^{\\bL^\bdf_\bx(t-s)} ds = \bL^\bdf_\bu  \frac{e^{\\bL^\bdf_\bx t} - 1}{\\bL^\bdf_\bx} \epsilon.
\]
\end{proof}
The stability estimate in Lemma~\ref{lemma:ODE_controlPerturb} allows us to quantify how errors in approximating the control function translate into trajectory deviations and, consequently, into errors in the cost functional, leading to the following approximation error analysis.
\begin{lemma}[Approximation error]\label{lemma:approx}
For any $\epsilon \in (0,1)$, there exists a ReLU network class $\mathscr{G}(R, \kappa, l, p, K)$ with $L = \widetilde{O}(\log \frac{1}{\epsilon})$, $p=\widetilde{O}(\epsilon^{-\frac{d}{\alpha+s}})$ and $K = \widetilde{O}\!\left(\epsilon^{-\tfrac{d}{s+\alpha}} 
\log \tfrac{1}{\epsilon} + (D+k) \log \tfrac{1}{\epsilon}\right)$, such that for any $\mG^*$, 
there exists a choice of weight parameters yielding a neural operator $\mG_{\hat\theta}$ such that 
\begin{equation}
    \mR(\mG_{\hat\theta}) -  \mR(\mG^*) \leq \widetilde{C}_{\rm III} \epsilon,
\end{equation}
where $\widetilde{C}_{\rm III}$ is a constant depends on $\bL^M, \bL^{\bdf},\bL^L,C_{\bdf}$ and $T$.
\end{lemma}

\begin{proof}
Theorem~1 in \cite{chen2022nonparametric} provides an approximation error estimate for scalar functions using ReLU networks. Our target, however, is a $q$-dimensional function,
\[
\mG^*(B,\bx_0)(t) = \bigl(g_1^*(B,\bx_0,t), \ldots, g_q^*(B,\bx_0,t)\bigr)^\top \in \mathbb{R}^q.
\]
Since each component $g_i^*$ can be approximated by a ReLU network in $\mathscr{F}(R,\kappa,l,p,K)$, we may use $q$ such networks to construct a vector-valued network in $\mathscr{G}(q, R,\kappa,l,p,K)$. Therefore, for any $\epsilon \in (0,1)$, there exists $\mG_{\hat\theta} \in \mathscr{G}(q, R,\kappa,l,p,K)$ such that each component satisfies
\[
\| g_i^*(B,\bx_0,t) - g_{\hat\theta,i}(B,\bx_0,t) \|_{\infty} \le \epsilon,
\qquad i = 1,\ldots,q.
\]
Consequently,
\begin{equation}
    \|\mG_{\hat\theta}(B,\bx_0)(t) - \mG^*(B,\bx_0)(t)\|_{\infty,\infty}
    = \max_{1\le i\le q} 
    \| g_{\hat\theta,i}(B,\bx_0,t) - g_i^*(B,\bx_0,t) \|_{\infty}
    \le \epsilon.
\end{equation}
And such network has at most $K \leq c_2 (\epsilon^{-(d+k)/n} \log \frac{1}{\epsilon} + D\log \frac{1}{\epsilon} + D \log D)$ neurons and weight parameters. Assume that 
\begin{equation}
\{\bx_{\hat{\theta}}(t)~|~ t\in [0,T]\} \text{ satisfy} 
\left\{
\begin{aligned}
    \dot{\mathbf{x}}_{\hat{\theta}}(t) &= \bdf(t, \mathbf{x}_{\hat{\theta}}(t); \mG_{\hat\theta}(B,\bx_0)(t)), \quad t \in [0, T] \\
    \mathbf{x}_{\hat{\theta}}(0) &= \mathbf{x}_0
\end{aligned}\right.
\end{equation}
and 
\begin{equation}
    \{\bx^*(t)~|~ t\in [0,T]\} \text{ satisfy }0 
\left\{
\begin{aligned}
    \dot{\bx}^*(t) &= \bdf(t, \bx^*(t); \mG^*(B, \bx_0)(t)), \quad t \in [0, T] \\
    \bx^*(0) &= \mathbf{x}_0
\end{aligned}\right.
\end{equation}
Then, for any $\bx_0\sim \mX, B\sim \mM $ we have 
\begin{align}
|J_B(\mathbf{x}^*(t), \mG^*(B,\bx_0))  -& J_B(\bx_{\hat{\theta}}(t), \mG_{\hat\theta}(B,\bx_0))| \\& \leq   \int_0^T |L(t, \bx_{\hat{\theta}}(t), \mG_{\hat\theta}(B,\bx_0), B) - L(t, \bx^*(t), \mG^*(B,\bx_0), B)| dt + |M(\bx(T)) - M(\by(T))|\nonumber \\
& \leq \int_0^T \bL^L_\bx \|\bx_{\hat{\theta}}(t) - \bx^*(t)\|_\infty dt + \int_0^T \bL^L_\bu \epsilon dt + \int_0^T \bL^B \|\bx_{\hat{\theta}}(t) - \bx^*(t)\|_\infty dt  + \bL^M\|\bx_{\hat{\theta}}(T) - \bx^*(T)\|_\infty \nonumber \\
&\leq \int_0^T \bL^L_\bx  C_\bdf (e^{\bL^\bdf_\bx t} - 1) \epsilon dt + \int_0^T \bL^L_\bu \epsilon dt + \int_0^T \bL^B  C_\bdf (e^{\bL^\bdf_\bx t} - 1) \epsilon dt  + \bL^M  C_\bdf (e^{\bL^\bdf_\bx T} - 1) \epsilon \nonumber \\
& \leq \widetilde{C}_{\rm III} \epsilon ,
\end{align}
where $\widetilde{C}_{\rm III}$ is a constant depends on $\bL^M, \bL^{\bdf},\bL^L,C_{\bdf}$ and $T$. From Jensen's inequality, we have
\begin{eqnarray}
    \mR(\mG_{\hat\theta}) -\mR(\mG^*)= |\mR(\mG_{\hat\theta}) -\mR(\mG^*) | \leq \mathbb{E}_{\bx_0\sim\rho} \mathbb{E}_{B\sim \mu} \Big| J_B(\mathbf{x}(t), \mG^*(B,\bx_0)) - J_B(\mathbf{x}(t), \mG_{\hat\theta}(B,\bx_0))   \Big|  \leq \widetilde{C}_{\rm III} \epsilon.
\end{eqnarray}
\end{proof}

\subsection{Statistic Error Term {\rm I}}
We first bound term {\rm I} in \eqref{eqn:errordecomp0}, which measures the deviation between the population risk and the empirical risk evaluated at the empirical minimizer $\mG_{\Gamma;\theta^*}$. We introduce the following notation. Given a training data pair $(B,\bx_0)$, the network $\mG_{\theta}$ produces a trajectory via neural ODE $\mathcal{T}(B,\bx_0,\mG_\theta) = \{ \bx(t) ~|~ t\in [0,T]\}$ satisfying
\begin{equation}
    \left\{
\begin{aligned}
    \dot{\mathbf{x}}(t) &= \bdf(t, \mathbf{x}(t); \mG_{\theta}(B,\bx_0)(t)), \quad t \in [0, T] \\
    \mathbf{x}(0) &= \mathbf{x}_0
\end{aligned}\right.
\end{equation}
We further define the composed functional $\mF\circ \mT(B,\bx_0,\mG_\theta) := J_B(\bx(t),\mG_\theta(B,\bx_0))$. The covering number estimation of the function class $(\mF\circ \mT)_\theta$ can then be estimated by the following Lemma~\ref{lemma:covering}, which is a direct consequence of Lemma 5 in \cite{chen2022nonparametric} together with Assumption~\ref{assump:Lip_cont}. 
\begin{lemma}[Covering number estimation]\label{lemma:covering}
\label{lemma:convernumber}
For $\mG_\theta\in\mathscr{G}(q, R, \kappa, l, p, K)$, given $\delta>0$, we have the following convering number estimation:
\begin{equation}
\mathcal{N}(\delta,(\mF\circ\mathcal{T})_\theta,\|\cdot\|_{L^\infty, \infty}) \leq \mathcal{N}(\frac{\delta}{\bL^J}, \mG_\theta,\|\cdot\|_{L^\infty(t, B, \bx_0), \infty}) \leq  \left( \frac{ 2 l^{2} (pE + 2) \, \kappa^{l} \, p^{l+1} }{ \delta / \bL^J } \right)^{qK}.
\end{equation}
\end{lemma}
Combining this estimate with Assumption~\ref{assump:bound_cost}, we obtain Lemma~\ref{lemma:term_I}, which establishes a bound for the statistical error term~{\rm I} in the decomposition~\eqref{eqn:errordecomp0}.
\begin{lemma}[Estimation of error term {\rm I} in \eqref{eqn:errordecomp0}]\label{lemma:stat_I}\label{lemma:term_I}
For $\mG_\theta\in\mathscr{G}(q, R, \kappa, l, p, K)$, under the \cref{assump:bound_cost}, that is $|J_B(\bx(t),\mG_\theta(B,\bx_0))| \leq R^J$ uniformly, given a training set $\Gamma = \{(B^i,\bx_0^i)\}_{i=1}^n \in \mX\times \mM$, we have
    \begin{equation}%\label{eqn:Rademarcher}
    \mathbb{E}_\Gamma (\mR(\mG_{\Gamma;\theta^*}) - \mR_\Gamma(\mG_{\Gamma;\theta^*})) \leq \frac{R^J}{\sqrt{n}}  \sqrt{qK\log(\frac{2 l^2 (pE + 2) \, \kappa^{l} \, p^{l+1} \bL^J \sqrt{n}}{R})}.
\end{equation}
Further, for any $\epsilon > 0$, let $l = \widetilde{O}(\log \frac{1}{\epsilon})$, $p=\widetilde{O}(\epsilon^{-\frac{d+k}{\alpha+s}})$, $K = \widetilde{O}\!\left(\epsilon^{-\tfrac{d+k}{s+\alpha}} \log \tfrac{1}{\epsilon} + D \log \tfrac{1}{\epsilon}\right)$ and $\kappa = O(\max \{ 1, E, \sqrt{d+k}, \tau^2 \})$, we then have 
\begin{equation}
    \mathbb{E}_\Gamma (\mR(\mG_{\Gamma;\theta^*}) - \mR_\Gamma(\mG_{\Gamma;\theta^*})) \leq 
\frac{\widetilde{C}_{\rm I}}{\sqrt{n}} \sqrt{(\epsilon^{-\tfrac{d+k}{s+\alpha}} + D) \log^3 \frac{1}{\epsilon}},
\end{equation}
Where $\widetilde{C}_{\rm I}$ is a constant depends on  $R^J, E, \bL^J, \log D, d, k, \tau$ and $s$.
\end{lemma}

\begin{proof}
    The error term {\rm I} in \eqref{eqn:errordecomp0} can be bounded by 
\[ \mathbb{E}_\Gamma (\mR(\mG_{\Gamma;\theta^*}) - \mR_\Gamma(\mG_{\Gamma;\theta^*})) \leq \mathbb{E}_\Gamma \sup_{\mG_\theta\in \mathscr{G}}(\mR(\mG_{\theta}) - \mR_\Gamma(\mG_{\theta})).  \]
Based on Lemma 26.2 in \cite{shalev2014understanding}, we have
\[  \mathbb{E}_\Gamma \sup_{\mG_\theta\in \mathscr{G}}(\mR(\mG_{\theta}) - \mR_\Gamma(\mG_{\theta})) \leq 2\mathbb{E}_\Gamma [\mathrm{Rad}(\mF\circ\mT\circ\Gamma )].  \] 
Given a training set $\Gamma = \{(B^i,\bx_0^i)\}_{i=1}^n \in \mX\times \mM$, we define the Rademacher complexity of $\mF\circ\mT$ with respect to the training data set $\Gamma$ as
\begin{equation}\label{eqn:Rademarcher}
    \mathrm{Rad}(\mF\circ\mT\circ\Gamma) :=\frac{1}{n} \mathbb{E}_{\xi\sim \{\pm 1 \}^n } \left[\sup_{\mG_\theta\in\mathscr{G}} \sum_{i=1}^n \xi_i J_{B^i}(\bx^i(t),\mG_\theta(\bx_0^i,B^i))  \right].
\end{equation}
By Dudley entropy integral \cite{dudley1967sizes}, we have
\begin{equation}\label{eq:RADbound}
        \mathbb{E}_\Gamma [\mathrm{Rad}(\mF\circ\mT\circ\Gamma )] \leq \inf_{ \sigma>0}\left(2\sigma+\frac{12}{\sqrt{n}}\int_{\sigma}^{R^J}\sqrt{\log\mathcal{N}(\delta,\mF\circ\mathcal{T},\|\cdot\|_{L^\infty})}\text{ }d\delta\right).
\end{equation}
By \cref{lemma:covering}, we have
\begin{align}
    \mathbb{E}_\Gamma [\mathrm{Rad}(\mF\circ\mT\circ\Gamma )] & \leq \inf_{ \sigma>0}\left(2\sigma+\frac{12}{\sqrt{n}}\int_{\sigma}^{R^J}\sqrt{\log\mathcal{N}(\delta,\mF\circ\mathcal{T},\|\cdot\|_{L^\infty})}\text{ }d\delta\right) \\
    & 
    \leq  \inf_{ \sigma>0}\left(2\sigma+\frac{12}{\sqrt{n}}\int_{\sigma}^{R^J}\sqrt{qK \log \left( \frac{ 2 l^{2} (pE + 2) \, \kappa^{l} \, p^{L+1} \bL^J}{ \delta } \right)}\text{ }d\delta\right) \\
    & \leq \inf_{ \sigma>0}\left(2\sigma+\frac{12}{\sqrt{n}} R^J \sqrt{qK \log \left( \frac{ 2 l^{2} (pE + 2) \, \kappa^{l} \, p^{l+1} \bL^J}{ \sigma} \right)}\right) \label{ieq:1} \\
    & \leq \frac{R^J}{\sqrt{n}} \left( 12  \sqrt{qK\log(\frac{2 l^{2} (pE + 2) \, \kappa^{l} \, p^{l+1} \bL^J\sqrt{n}}{R})} + 2 \right).\label{ieq:2}
\end{align}
Note \eqref{ieq:2} is obtained by using $\sigma = \frac{R^J}{\sqrt{n}}$ in \eqref{ieq:1}. Since $l = \widetilde{O}(\log \frac{1}{\epsilon})$, $p=\widetilde{O}(\epsilon^{-\frac{d}{\alpha+s}})$ and $K = \widetilde{O}\!\left(\epsilon^{-\tfrac{d}{s+\alpha}} \log \tfrac{1}{\epsilon} + D \log \tfrac{1}{\epsilon}\right)$, we have
\[
\log l = \widetilde{O}(\log \log \frac{1}{\epsilon}), ~ \log(pE + 2) = \widetilde{O}(\log \frac{1}{\epsilon}), ~ l \log \kappa = \widetilde{O}(\log \frac{1}{\epsilon} \log \kappa),~ (l+1) \log p = \widetilde{O}(\log^2 \frac{1}{\epsilon}),
\]
which implies
\begin{align}
    \log(\frac{2 l^{2} (pE + 2) \, \kappa^{l} \, p^{l+1} \bL^J\sqrt{n}}{R^J}) &=  2\log l+\log(pE+2)+l\log \kappa + (l+1)\log p + \log \sqrt{n} + \log \frac{\bL^J}{R^J}, \nonumber \\
    &= \widetilde{O}(\log n + \log^2 \frac{1}{\epsilon}), \label{eqn:log_1st_term}
\end{align}
and
\begin{equation}\label{eqn:log_2nd_term}
    {K\log(\frac{2 l^{2} (pE + 2) \, \kappa^{l} \, p^{l+1} \bL^J\sqrt{n}}{R^J})} = \widetilde{O}\left(\log \tfrac{1}{\epsilon}(\epsilon^{-\tfrac{d}{s+\alpha}} + D)  (\log n + \log^2 \frac{1}{\epsilon} )\right).
\end{equation} 
By plugging \eqref{eqn:log_1st_term} and \eqref{eqn:log_2nd_term} into \eqref{ieq:2}, we have
\[
\frac{R^J}{\sqrt{n}} \left( 12  \sqrt{qK\log(\frac{2 l^{2} (pE + 2) \, \kappa^{l} \, p^{l+1} \bL^J\sqrt{n}}{R})} + 2 \right) \leq  
\frac{\widetilde{C}_{\rm I}}{\sqrt{n}} \sqrt{(\epsilon^{-\tfrac{d}{s+\alpha}} + D) \log^3 \frac{1}{\epsilon}}
\]
where $\widetilde{C}_{\rm I}$ depends on $q, R^J, E, \bL^J, \log D, d, k, \tau$ and $s$.
\end{proof}

\subsection{Statistic Error Term {\rm II}}
We now turn to the analysis of the statistical error term~{\rm II} in the decomposition~\eqref{eqn:errordecomp0}. Using concentration inequalities, we obtain the following bound.
\begin{lemma}[Statistic error term {\rm II} in \eqref{eqn:errordecomp0}]\label{lemma:stat_II}
\[ \mathbb{E}_\Gamma (\mR(\mG_{\Gamma;\theta^*}) - \mR_\Gamma(\mG_{\Gamma;\theta^*})) \leq (1-\epsilon)\sqrt{ \frac{2\log(1/\epsilon)}{n}} + 2 \epsilon R^J.  \]
\end{lemma}

\begin{proof}
Note that 
\[ \mathbb{E}_\Gamma (\mR(\mG_{\Gamma;\theta^*}) - \mR_\Gamma(\mG_{\Gamma;\theta^*})) \leq \mathbb{E}_\Gamma \sup_{\mG_\theta\in \mathscr{G}}(\mR(\mG_{\theta}) - \mR_\Gamma(\mG_{\theta})).  \]
Based on Hoeffding's inequality, since $|J_B(\bx_0,\mG_\theta(B,\bx_0))| \leq R^J$ uniformly, we have, for any $t>0$
\begin{eqnarray}
    \mathbb{P}(\mR_\Gamma(\mG_{\hat\theta})- \mR(\mG_{\hat\theta}) \geq t )  \leq \displaystyle e^{-\frac{ t^2n}{2{(R^J)}^2}}.
\end{eqnarray}
Let $\epsilon = e^{-\frac{ t^2n}{2{(R^J)}^2}}$, we have that with probability $1-\epsilon $,
\begin{equation}
    \mR_\Gamma(\mG_{\hat\theta})- \mR(\mG_{\hat\theta}) \leq R^J \sqrt{ \frac{2\log(1/\epsilon)}{n}}.
\end{equation}
Therefore, we have
\begin{equation}\label{ieq:}
    \mathbb{E}_\Gamma [\mR_\Gamma(\mG_{\hat\theta})- \mR(\mG_{\hat\theta}) ] \leq  (1-\epsilon)\sqrt{ \frac{2\log(1/\epsilon)}{n}} +  \epsilon (2R^J).
\end{equation}
\end{proof}

\subsection{Proof of \Cref{thm:gen}}
\begin{proof}
To prove the theorem, we first decompose the generalization error into three terms. Given a training set $\Gamma = \{(B^i,\bx_0^i)\}_{i=1}^n \in \mM\times \mX$, note that the empirical risk of the optimal operator in $\mathscr{G}$ using sample $\Gamma$ satisfies $\mR_\Gamma(\mG_{\Gamma;\theta^*}) = \min_{\mG_\theta\in\mathscr{G}} \mR_\Gamma(\mG) \leq  \mR_\Gamma(\mG_{\hat\theta})$, we then decompose the generalization error as:
\begin{eqnarray}
        0 \leq \mR(\mG_{\Gamma;\theta^*}) - \mR(\mG^*) 
     &\leq &     \mR(\mG_{\Gamma;\theta^*}) - \mR_\Gamma(\mG_{\Gamma;\theta^*}) + 
    \mR_\Gamma(\mG_{\hat\theta})- \mR(\mG_{\hat\theta})   + 
     \mR(\mG_{\hat\theta}) - \mR(\mG^*) 
\end{eqnarray}
Therefore, we have
\begin{eqnarray}\label{eqn:errordecomp_eps}
   \mathbb{E}_\Gamma (\mR(\mG^*) - \mR(\mG_{\Gamma;\theta^*})) \leq 
    \underbrace{\mathbb{E}_\Gamma(\mR(\mG_{\Gamma;\theta^*}) - \mR_\Gamma(\mG_{\Gamma;\theta^*}))}_{{\rm I: statistic error}} + 
    \underbrace{\mathbb{E}_\Gamma (\mR_\Gamma(\mG_{\hat\theta})- \mR(\mG_{\hat\theta}))}_{{\rm II: statistic error}} + 
        \underbrace{(\mR(\mG_{\hat\theta}) - \mR(\mG^*))}_{{\rm III: approximation error}}
\end{eqnarray}

For any $\epsilon \in (0,1)$, from Lemma~\ref{lemma:stat_I}, there exists a neural operator $\mG_{\Gamma;\theta^*}$ in ReLU network class $\mathscr{G}(q, R, \kappa, l, p, K)$ with $l = \widetilde{O}(\log \frac{1}{\epsilon})$, $p=\widetilde{O}(\epsilon^{-\frac{d+k}{\alpha+s}})$, $K = \widetilde{O}\!\left(\epsilon^{-\tfrac{d+k}{s+\alpha}} \log \tfrac{1}{\epsilon} + (D+k) \log \tfrac{1}{\epsilon}\right)$ and $\kappa = O(\max \{ 1, E, \sqrt{d+k}, \tau^2 \})$, such that
\begin{equation}\label{eqn:stat_I_eps} 
\mathbb{E}_\Gamma (\mR(\mG_{\Gamma;\theta^*}) - \mR_\Gamma(\mG_{\Gamma;\theta^*})) \leq 
\frac{\widetilde{C}_{\rm I}}{\sqrt{n}} \sqrt{(\epsilon^{-\tfrac{d+k}{s+\alpha}} + D) \log^3 \frac{1}{\epsilon}}.
\end{equation}
% \begin{equation}\label{eqn:stat_I_eps} 
%     0\leq \mathbb{E}_\Gamma [ \mR(\mG_{\Gamma;\theta^*}) - \mR(\mG^*)]  \leq \frac{\widetilde{C}_{\rm I}}{\sqrt{n}} \sqrt{(\epsilon^{-\tfrac{d+k}{s+\alpha}} + D+k) \log^3 \frac{1}{\epsilon}} + (1-\epsilon)\sqrt{ \frac{2\log(1/\epsilon)}{n}} + (\widetilde{C}_{\rm III} + 2R) \epsilon.
% \end{equation} 
For the same choice of $\epsilon$, Lemma~\ref{lemma:stat_II} yields
\begin{equation}\label{eqn:stat_II_eps} 
\mathbb{E}_\Gamma (\mR(\mG_{\Gamma;\theta^*}) - \mR_\Gamma(\mG_{\Gamma;\theta^*})) \leq (1-\epsilon)\sqrt{ \frac{2\log(1/\epsilon)}{n}} + 2\epsilon R. 
\end{equation}
Moreover, Lemma~\ref{lemma:approx} provides
\begin{equation}\label{eqn:approx_eps}
    \mR(\mG_{\hat\theta}) -  \mR(\mG^*) \leq \widetilde{C}_{\rm III}  \epsilon.
\end{equation}
By plugging \eqref{eqn:stat_I_eps}, \eqref{eqn:stat_II_eps} and \eqref{eqn:approx_eps} into \eqref{eqn:errordecomp_eps}, we obtain
\begin{equation}
        \mathbb{E}_\Gamma [ \mR(\mG_{\Gamma;\theta^*}) - \mR(\mG^*)]  \leq \frac{\widetilde{C}_{\rm I}}{\sqrt{n}} \sqrt{(\epsilon^{-\tfrac{d+k}{s+\alpha}} + D+k) \log^3 \frac{1}{\epsilon}} + (1-\epsilon)\sqrt{ \frac{2\log(1/\epsilon)}{n}} + (\widetilde{C}_{\rm III} + 2R) \epsilon.
        %\leq     
        %\frac{R}{\sqrt{n}}  \sqrt{K\log(\frac{2 L^{2} (pB + 2) \, \kappa^{L} \, p^{L+1} \bL^J \sqrt{n}}{R})} + (1-\epsilon)\sqrt{ \frac{2\log(1/\epsilon)}{n}} + C_1 \epsilon.
\end{equation}

By picking $\displaystyle \epsilon^2 = \frac{1}{n}{\epsilon^{-\tfrac{d+k}{s+\alpha}}}$, i.e. $\epsilon = n^{-\frac{s+\alpha}{d+k+2(s+\alpha)}}$, then the size of the ReLU network becomes $L = \widetilde{O}(\frac{s+\alpha}{d+k+2(s+\alpha)} \log n)$, $p=\widetilde{O}(n^{\frac{d+k}{2(\alpha+s)+d+k}})$ and $K = \widetilde{O}\left(\frac{s+\alpha}{2(s+\alpha)+d+k} \, n^{\tfrac{d+k}{2(s+\alpha)+d+k}} \log n\right)$, and we have
\begin{align*}
    \frac{\widetilde{C}_{\rm I}}{\sqrt{n}} \sqrt{\left(\epsilon^{-\tfrac{d+k}{s+\alpha}} + D+k \right) \log^3 \frac{1}{\epsilon}} & = \widetilde{C}_{\rm I} \sqrt{(\frac{1}{n} \epsilon^{-\tfrac{d+k}{s+\alpha}} + \frac{D+k}{n}) \log^3 \frac{1}{\epsilon}} \\
    & = \widetilde{C}_{\rm I} \sqrt{\left(n^{-\frac{2(s+\alpha)}{d+k+2(s+\alpha)}}+\frac{D+k}{n}\right)\left(\frac{s+\alpha}{d+k+2(s+\alpha)}\right)^3 \log^3n} \\
    & = \widetilde{C}_{\rm I} \left(\frac{s+\alpha}{d+k+2(s+\alpha)}\right)^{\tfrac{3}{2}}   \sqrt{n^{-\frac{2(s+\alpha)}{d+k+2(s+\alpha)}}+\frac{D+k}{n}}  \log^{\tfrac{3}{2}}n.
\end{align*}
Thus
\[
\displaystyle \mathbb{E}_\Gamma [ \mR(\mG_{\Gamma;\theta^*}) - \mR(\mG^*)]  \leq C_1 \sqrt{n^{-\frac{2(s+\alpha)}{d+k+2(s+\alpha)}}+\frac{D+k}{n}}\log^{\frac{3}{2}}n + C_2 \sqrt{\frac{\log n}{n}} + C_3 n^{-\frac{s+\alpha}{d+k+2(s+\alpha)}},
\]
where $C_1 = \widetilde{C}_{\rm I} \left(\frac{s+\alpha}{d+k+2(s+\alpha)}\right)^{\tfrac{3}{2}} $, $C_2 = \sqrt{\frac{2(s+\alpha)}{d+k+2(s+\alpha)}}$ and $C_3=\widetilde{C}_{\rm III} + 2R$.
\end{proof}

\section{Extension to Closed-Loop Control}\label{sec:extension}

Our solution operator approach maps a given initial state and environment function $B(\bx)$ to an \emph{open-loop} optimal control function, which is effective when the environment is static or its dynamics are fully known in advance. However, in many real-world applications, the environment function $B(\bx)$ is not fully known in advance. For instance, consider a robot delivery task in which the robot starts at a designated location and must deliver a package to a target position, while the environment contains unknown obstacles representing people that may appear unexpectedly along its path. The robot has limited vision and can only detect these obstacles once they enter its sensing range; before that moment, it has no knowledge of their existence. As the robot approaches such obstacles, their sudden appearance within its field of view requires an immediate and safe reaction to avoid collisions. Such tasks are therefore fundamentally \emph{closed-loop} optimal control problem in nature: the control policy must continuously adapt to feedback from the actual, evolving environment and the current state.

To handle such scenarios, we extend our approach to a closed-loop setting by integrating the learned solution operator with a Model Predictive Control (MPC) framework \cite{camacho2007constrained,rawlings2020model}. MPC transforms open loop optimal control into closed loop behavior by repeatedly solving a short horizon optimal control problem and applying only the initial segment of the computed control. Although each solve is open loop, the continual re-optimization at every time step induces an implicit feedback mechanism. A limitation of conventional MPC, however, is that even short horizon optimal control problems can be computationally costly to solve in real time. This is precisely where the proposed solution operator framework offers a significant advantage: the learned operator replaces the expensive optimization at each MPC step, enabling rapid closed loop control with substantially reduced computational burden.

First, we reformulate the optimal control problem as a free-terminal-time problem, while keeping the underlying dynamics unchanged.
\begin{equation*}
{ \mathbf{u}^*(t)}, { T^*}  \in  \arg \min_{\bu, T} J_B(\mathbf{x}(t), \bu) = w_1 \int_0^T L_B(t, \bx(t), \bu(t)) \rd t + w_2  M(\bx(T)) + w_3 { T}
\end{equation*}
\begin{equation}\label{eqn:dynamics_2}
\left\{
\begin{aligned}
    \dot{\mathbf{x}}(t) &= \bdf(t, \mathbf{x}(t), \mathbf{u}(t)), \quad t \in [0, T] \\
    \mathbf{x}(0) &= \mathbf{x}_0.
\end{aligned}
\right.
\end{equation}
Here $L_B$ and $M$ are the same running and terminal cost terms as in \eqref{eqn:Loss}, while the terminal time $T>0$ is an additional decision variable. The reason of introducing this additional decision variable $T$ are two-fold. First, in contrast to classical MPC, which repeatedly solves a finite-horizon problem over a fixed time interval, our goal is to endow the algorithm with a global view of the task by learning the full optimal trajectory to the target. Allowing $T$ to be optimized encourages solutions that reflect the overall optimal time and path length for reaching the target, rather than being tied to a fixed finite horizon. Second, a fixed terminal time becomes inappropriate as the state approaches the target during inference, since it forces the predicted control to spread the motion over the entire time horizon, which can lead to unnecessarily slow trajectories. By treating $T$ as a decision variable and penalizing its magnitude in the cost, we naturally promote fast arrival when possible.

In this setting, the solution-operator learning task is a straightforward extension of the original formulation: rather than learning only the optimal control function, the neural operator is now trained to approximate the mapping $\mG: (B, \bx_0, \bx_{T}) \mapsto  (\mathbf{u}^*(t), T^*)$, which maps the obstacle field, initial condition and terminal condition to both the optimal control and the corresponding terminal time. 
After that, we embed the learned operator $\mG_{\Gamma;\tilde{\theta}}$ into a MPC framework using the learned solution operator to predict each short horizon optimal control problem: at each step, we use $\mG_{\Gamma;\tilde{\theta}}$ to map the current state and currently visible obstacles to a control function on this interval, and then obtain the corresponding trajectory by integrating the system dynamics. The procedure is detailed Algorithm~\ref{alg:mpc_inference}. %This constant feedback–replanning algorithm enables the robot to react safely and immediately to suddenly appearing obstacles while making steady progress toward the goal. 
To demonstrate the effectiveness of this approach, Section~\ref{sec:MPC} presents numerical results for a unicycle robot navigating a maze-like environment with unknown objects that may appear unexpectedly along the robot's path.

\begin{algorithm}[htp]
\caption{MPC with Learned Solution Operator under Partial Observability}\label{alg:mpc_inference}
\begin{algorithmic}[1]
\Require Learned operator $\mathcal{G}_{\Gamma;\tilde{\theta}^*}$, initial state $\mathbf{x}_0$, target $\mathbf{x}_T$, timestep $\Delta t$, vision radius $R$, tolerance $\delta$
\State Set $k = 0$ and $\mathbf{x}(0) = \mathbf{x}_0$

\While{$\|\mathbf{x}(k\Delta t) - \mathbf{x}_T\| > \delta$}
    \State{Update the current obstacle function, denoted as $B_k(\bx)$}

    \State Obtain control function
    \[
        (\mathbf{u}_k(t),\, T_k) 
        = 
        \mathcal{G}_{\Gamma;\tilde{\theta}^*}\big(B_k,\; \mathbf{x}(k\Delta t), \bx_T \big)
    \]
    
    \State {Evolve the dynamics} according to~\eqref{eqn:uni_dyn} using control $\mathbf{u}_k(t)$ and obtain
    \[
       \mathbf{x}(t), \qquad t \in [k\Delta t,\,(k+1)\Delta t].
    \]

    \State $k \gets k + 1$

\EndWhile
\State Set $K = k$
\State \Return trajectory $\{\mathbf{x}(k\Delta t)\}_{k=0}^{K}$
\end{algorithmic}
\end{algorithm}
We note that this construction that combines neural operator with MPC only produces an approximation to the \emph{closed-loop} control control, and the fundamental convergence issues inherent in MPC still persist and are out of the scope of this paper.

\section{Numerical Experiments}\label{sec:numerics}
In this section, we present three numerical examples to demonstrate the effectiveness of the proposed method. At the same time, these examples serve to numerically validate our scaling law theory: the method cannot perform universally well across problems of varying complexity. As the complexity of the control task increases, it becomes increasingly challenging to learn an effective operator. We begin with stationary and moving Gaussian-bump barriers in both 2D and 3D settings, which illustrate the method’s effectiveness under different environmental variations. Next, we consider a more challenging maze navigation task, where the objective is to solve the maze while avoiding collisions with walls. Then, we extend our approach to the unicycle problem, which features nonlinear dynamics. For this case, we compare the performance of our method with a classical numerical solver based on nonlinear programming (NLP), showing that our approach achieves comparable cost minimization while providing nearly instantaneous inference for batch predictions. Finally, we examine the scaling laws of the stationary Gaussian-bump barrier and the maze navigation problem, through which we further verify the “price to pay” argument.

\textbf{Networks Architecture.} There are various approaches to approximating the solution operator $\mG^*$. For the purpose of theoretical error analysis, we model $\mG^*$ using deep ReLU networks. Since the  function $B$ is an infinite-dimensional function in $C(\Omega)$, it is intractable to handle it directly. In practice, one approximates $B$ by mapping it into a finite-dimensional ambient representation, for example through discretization or an encoder. For simplicity, in this paper we focus on the case that the barrier distribution $\mu$ is supported on a $k$-dimensional manifold $\mM \subset C(\Omega)$, where $k$ is the intrinsic dimension and usually much smaller than the ambient dimension. This simplified setting highlights the key message of our paper, as shown in~\Cref{thm:gen}.
%: the generalization error exponentially depends on the intrinsic dimension $k$, which indicates that one cannot avoid the curse of dimensionality with respect to $k$.
For the sake of clarity and analytical simplicity, we focus on stationary functions represented by Gaussian mixtures. The set of such barrier functions lies on a $k$-dimensional manifold, where $k$ corresponds to the number of free parameters in the Gaussian mixture. Specifically, let $N_b$ denote the number of mixture Gaussian bumps. Then the barrier function can be parameterized as
\[
B (\bx) = \sum_{j=1}^{N_b} h_j \mN(\bx; \boldsymbol{c}_j,\sigma^2_j I),
\]
where $\boldsymbol{c}_j \in \bR^{D}$ denotes the center of each bumps, $\mN$ denotes the density function of a standard Gaussian distribution, $h_j$ and $\sigma^2_j$ denotes the magnitude and variance. We represent the parameters of each component by the vector $\hat{B}_j:=[\mathbf{c}_j, h_j, \sigma_j^2]$. Consequently, the dimension of  $\mM$ is {$k=(D+2)N_b$}. %This intrinsic-dimension analysis can be naturally extended to the case of moving Gaussian mixtures as well as to other classes of barrier functions.

% For a given input $\mathbf{z} \in \mathbb{R}^{n_0}$, an $l$-layer,   ReLU neural network is defined as follows
% \begin{equation}\label{eqn:relu_def}
%     g^l_\theta(\bz) = W_l \cdot \operatorname{ReLU}\bigl( W_{l-1} \cdots \operatorname{ReLU}\bigl( W_1 \bz + \mathbf{b}_1 \bigr) \cdots + \mathbf{b}_{l-1} \bigr) + \mathbf{b}_l.
% \end{equation}
% For each hidden layer $i = 1, 2, \ldots, l-1$, the weight matrix and bias vector are denoted by  
% $W_i \in \mathbb{R}^{n_i \times n_{i-1}}$ and $\mathbf{b}_i \in \mathbb{R}^{n_i}$, respectively, where $n_i$ represents the width of layer $i$. The overall width of the network is defined as  
% \begin{equation}
%     p = \max_{i \in \{1, 2, \ldots, l-1\}} n_i,
% \end{equation}
% and all trainable parameters are collected in  
% $\theta = \{ W_1, \mathbf{b}_1, W_2, \mathbf{b}_2, \ldots, W_l, \mathbf{b}_l \}$. 
Since the order of the mixture components is arbitrary, the network must be permutation-invariant with respect to the $N_b$ bumps. To enforce this invariance, each component $\hat{B}_k$ is first embedded into a feature vector using a shared smaller ReLU network $g_2$. The resulting features are then aggregated via mean pooling (this step could alternatively be replaced by more sophisticated mechanisms such as self-attention without positional encoding). Finally, the aggregated barrier representation is concatenated with the global variables $(\bx_0, t)$ and passed through another ReLU network $g_1$ to obtain the prediction of the operator at time $t$. In particular, the resulting approximation network for $\mG^*(B, \bx_0)(t)$ takes the form
\begin{equation}\label{eqn:relu_net}
\mG_\theta(B, \bx_0)(t) = g_\theta^l(B, \bx_0, t ) = g^{l_1}_1 (\bx_0, t, \frac{1}{N_b} \sum_{k=1}^{N_b} g^{l_2}_2 (\hat{B}_k)).
\end{equation}
where $l=l_1+l_2$. For the Gaussian mixture and unicycle examples, we typically employ the deep ReLU network architecture defined in \eqref{eqn:relu_net}. For the maze navigation example, to better encode maze configurations, we enhance the model with Fourier random features and an additional multi-head self-attention mechanism. These architectures will be described in detail in the Appendix~\ref{sec:network_maze}.

\textbf{Data Generation.} The training and testing datasets of the Gaussian bumps and the unicycle problem are uniformly sampled from specific distributions, which will be described in detail in their respective sections.  The generation of the maze configurations follows the randomized Depth-First Search (DFS), which is detailed in Algorithm~\ref{alg:maze} in Appendix~\ref{sec:network_maze}.

\textbf{Hyperparameters.} The hyperparameters of the network architecture, including the number of layers and the number of neurons per layer, are summarized in Table~\ref{tab:hyper}. For the coupling dynamics, we use a neural ODE~\cite{chen2018neural} implemented in PyTorch~\cite{paszke2019pytorch}, discretized with the forward Euler scheme with terminal threshold stationary at $\delta=1e-4$. During training, we adopt the Adam optimizer~\cite{kingma2014adam}, with the learning rate $l_r$ and learning rate scheduler specified in Table~\ref{tab:hyper}.

\subsection{Gaussian Bump Barrier}
In this subsection, we numerically study optimal control problems with barrier functions parameterized by Gaussian mixtures. The state domain is $\Omega = [0,1]^D$, and we focus on the cases $D=2$ and $D=3$. Both stationary and moving barriers are considered. The admissible control set is $\mathcal{S} = \mathbb{R}^D$, and the running cost consists of two components: a kinetic energy term and a penalty term. The optimal control problem can then be formulated as:
\begin{equation}\label{eqn:Loss_gauss}
     \bu^*(t; B, \bx_0) \in  \arg \min_{\mathbf{u}} J_B(\mathbf{x}(t), \mathbf{u}) = \int_0^T \left( \frac{1}{2} |\bu(t, \bx(t))|^2 + B(\bx(t)) \right) \rd t +  w_t |\bx(T)|^2
\end{equation}
subject to the following dynamics
\begin{equation}\label{eqn:dynamics_gauss}
\left\{
\begin{aligned}
    \dot{\mathbf{x}}(t) &=  \mathbf{u}(t, \bx(t)), \quad t \in [0, T] \\
    \mathbf{x}(0) &= \mathbf{x}_0 .
\end{aligned}
\right.
\end{equation}
To generate the training set $\Gamma = \{B^i, \bx_0^i\}_{i=1}^n$, we uniformly sample initial states from $\mX$, where the selection of $\mX$ will be specified in each numerical experiment. For \textit{stationary barrier} function parameterized by gaussian mixture, we have
\begin{equation}\label{eqn:gauss}
    B(\bx) = \sum_{j=1}^{N_b} h_j \mN(\bx; \mathbf{c}_j, \sigma_j),
\end{equation}
where $\mathbf{c}_j \in \bR^D$ is the center of $j$-th gaussian bump, $\sigma_i$ is the isotropic standard deviation and $h_i$ is the height of each gaussian bump. Sample the parameters of each Gaussian bump from uniform distributions: $\mathbf{c}_j \sim U([c_1, c_2]^D)$, $\sigma_j \sim U[0.05, 0.1]$ and $h_j \sim U[0.1, 1]$. For \textit{moving barrier}, we consider a simplified setting in which each Gaussian bump undergoes pure linear translation without deformation. In this case, the barrier profile at time $t$ is given by
\begin{equation}\label{eqn:moving_gauss}
    B(t, \bx) = \sum_{j=1}^{N_b} h_j \mN(\bx; \mathbf{c}_j(t), \sigma_j),
\end{equation}
where the center of each bumps move linearly in $t$, $\mathbf{c}_j = \mathbf{c}_j^0 + \mathbf{v}_j t$, with initial state $\mathbf{c}^0_j \sim U([c_1, c_2]^D)$ and velocities $\mathbf{v}^0_j \sim U([v_1, v_2]^D)$.
\subsubsection{2D Stationary Gaussian Bumps}\label{sec:gauss_stationary}

\textbf{Parameter study with $N_b=1$.} First, we let $\mX = \big([0.5,1]\times[0.9,1]\big) \;\cup\; \big([0.9,1]\times[0.5,1]\big)$. We numerically investigate the case with a single Gaussian bump ($N_b=1$). We generate $n=500$ training data pairs as described above and train models separately with $w_t \in \{0.5, 1, 1.5, 2, 2.5 \}$. The trained models are then evaluated under the following scenarios. Across all cases, the optimal trajectory varies in the expected way with changes in the problem parameters.
\begin{itemize}
    \item[1.] \textit{Varying bump strength.} Fix the initial state, the bump center, and the standard deviation, and set $w_t=1$. We vary the bump height $h=0.1, 0.2, \dots, 1$ to study the effect of bump strength. The results are shown in the left panel of Figure~\ref{fig:2d_stationary_1b_decay}.
    \item[2.] \textit{Varying initial states.} With the bump center, standard deviation, and bump strengths stationary, we set $w_t=1$. We test the effect of different initial positions. The results are presented in the middle panel of Figure~\ref{fig:2d_stationary_1b_decay}.
    \item[3.] \textit{Varying terminal weight.} Fix the initial state, bump center position, standard deviation and bump strength. We then study the effect of the terminal cost weight by testing with models trained under different $w_t = 0.5, 1, 1.5, 2, 2.5$. The results are presented in the right panel of Figure~\ref{fig:2d_stationary_1b_decay}.
\end{itemize}
\begin{figure}[htbp]
    \centering \includegraphics[width=0.25\linewidth]{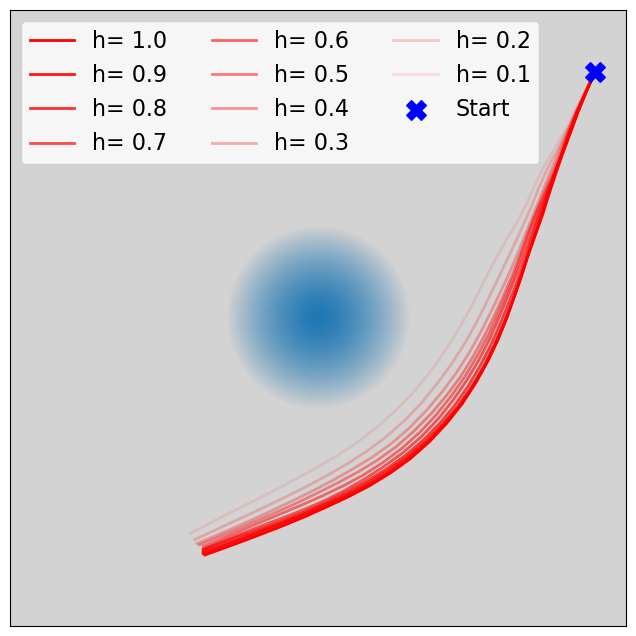}
    \includegraphics[width=0.25\linewidth]{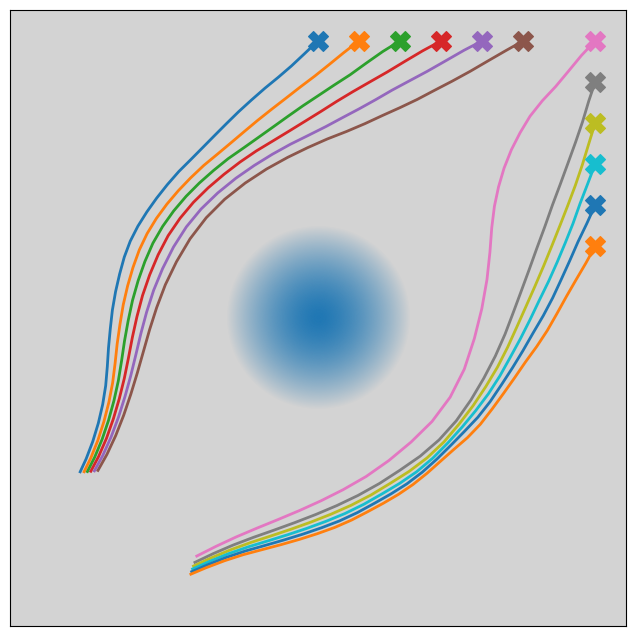}
    \includegraphics[width=0.25\linewidth]{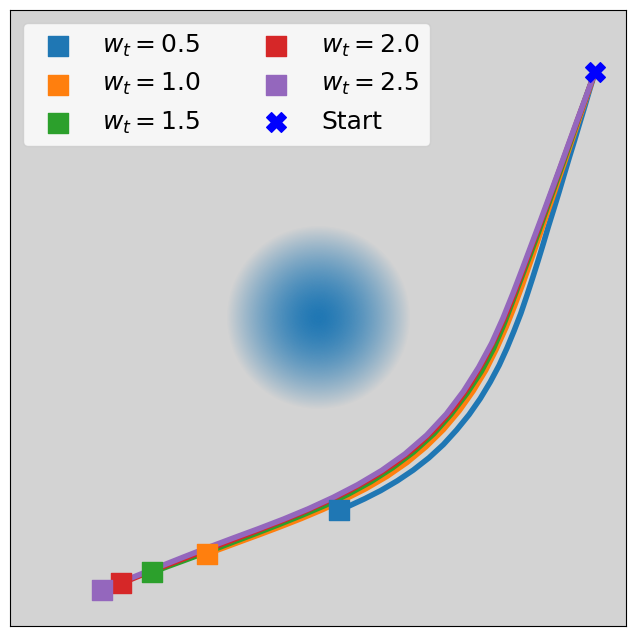}
    \caption{Stationary Gaussian-bump examples. Results are shown for varying bump heights/strengths (left), varying initial positions (middle), and varying terminal weights (right).}
    \label{fig:2d_stationary_1b_decay}
\end{figure}
\textbf{Parameter study on $N_b$.} In this part, we present numerical demonstrations with different numbers of Gaussian bumps, considering the cases $N_b = 2, 3$. Representative instances for these two settings are shown in Figure~\ref{fig:2d_stationary_1} and Figure~\ref{fig:2d_stationary_2}, respectively. As an illustrative example, we further examine how the model behaves when two bumps gradually transition from being fully overlapping to being well separated. This setup highlights how the model adapts its predictions under changing barrier configurations; for example, whether it chooses to avoid both bumps simultaneously when they overlap, or to navigate between them once they separate, and whether it can identify the shortest trajectory that minimizes kinetic energy.

% The corresponding numerical results are shown in Figure~\ref{fig:2d_stationary_bp2_0}.
% \begin{figure}[htbp]
%     \centering \includegraphics[width=0.6\linewidth]{figures/2D_stationary_adapting.png}
%     \caption{}
%     \label{fig:2d_stationary_bp2_0}
% \end{figure}

\begin{figure}[htbp]
    \centering
    \includegraphics[width=0.45\linewidth]{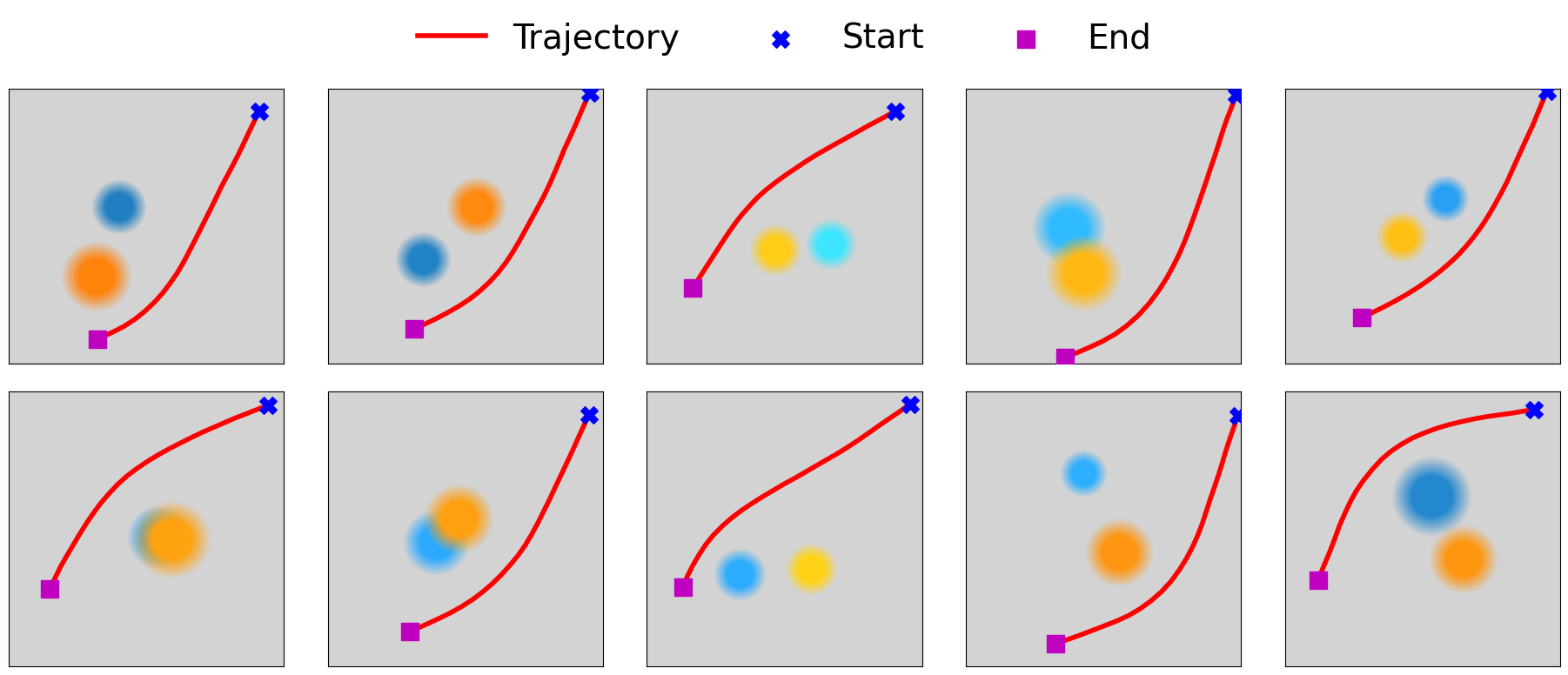}
    \hspace{3mm}
    \includegraphics[width=0.45\linewidth]{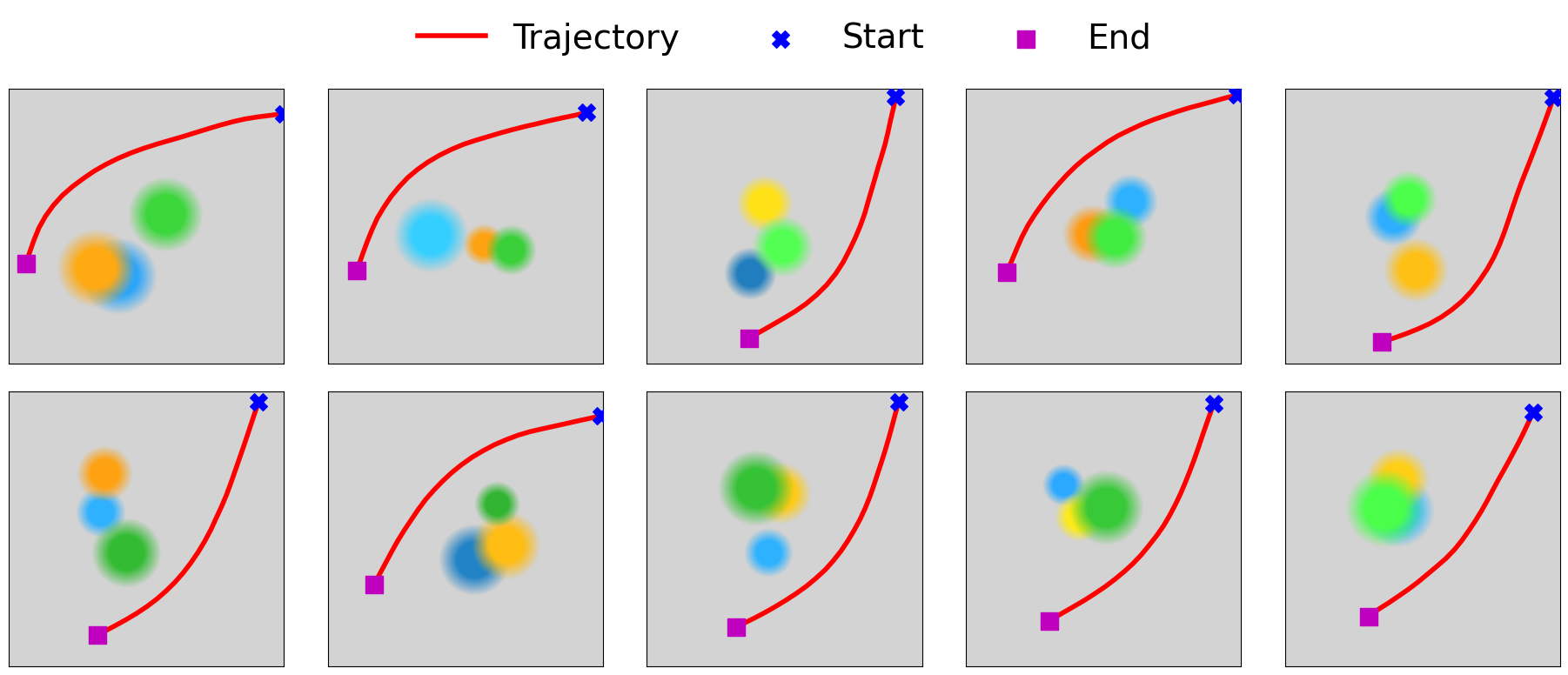}
    \caption{2D examples of trajectories in environments with stationary Gaussian-bump obstacles. The left panel shows 10 instances with two randomly placed bumps, and the right panel shows 10 instances with three. Red curves denote predicted trajectories from random starting points (blue crosses) to the terminal states at $T=3$ (purple square). The radius of each bump corresponds to 1.5 times its standard deviation, while brightness indicates the bump height $h$: darker shading represents larger $h$.}
    \label{fig:2d_stationary_1}
\end{figure}

\begin{figure}[htbp]
    \centering
    \includegraphics[width=0.45\linewidth]{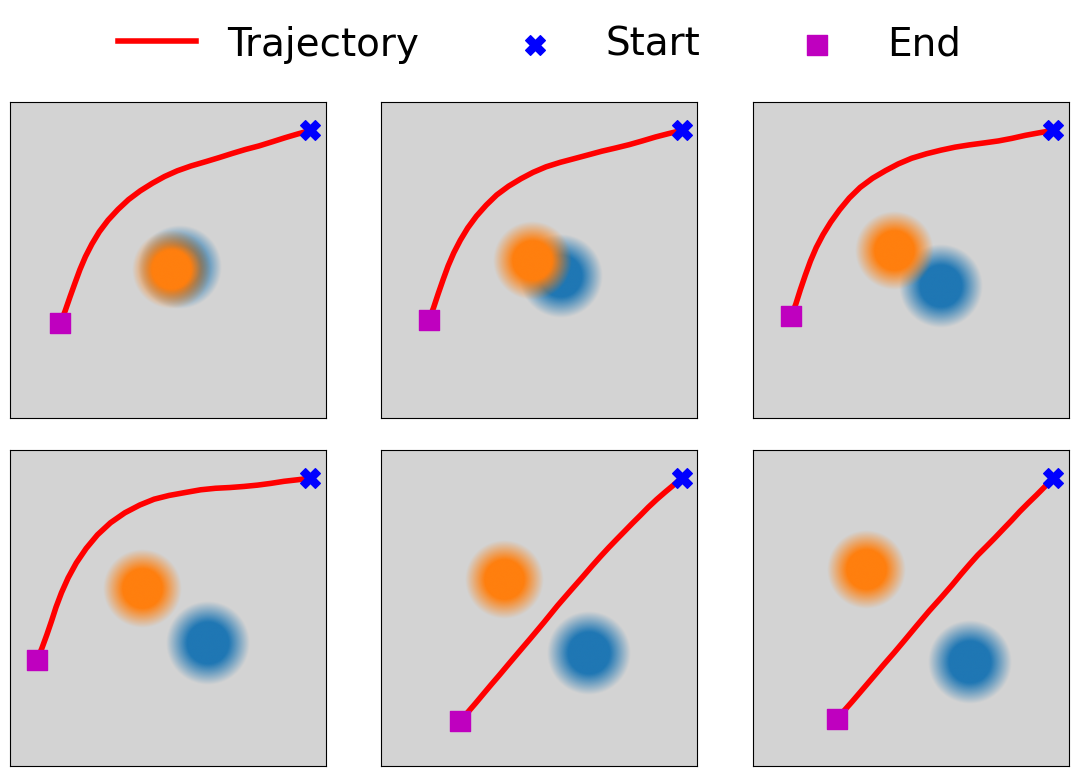}
    \hspace{3mm}
    \includegraphics[width=0.45\linewidth]{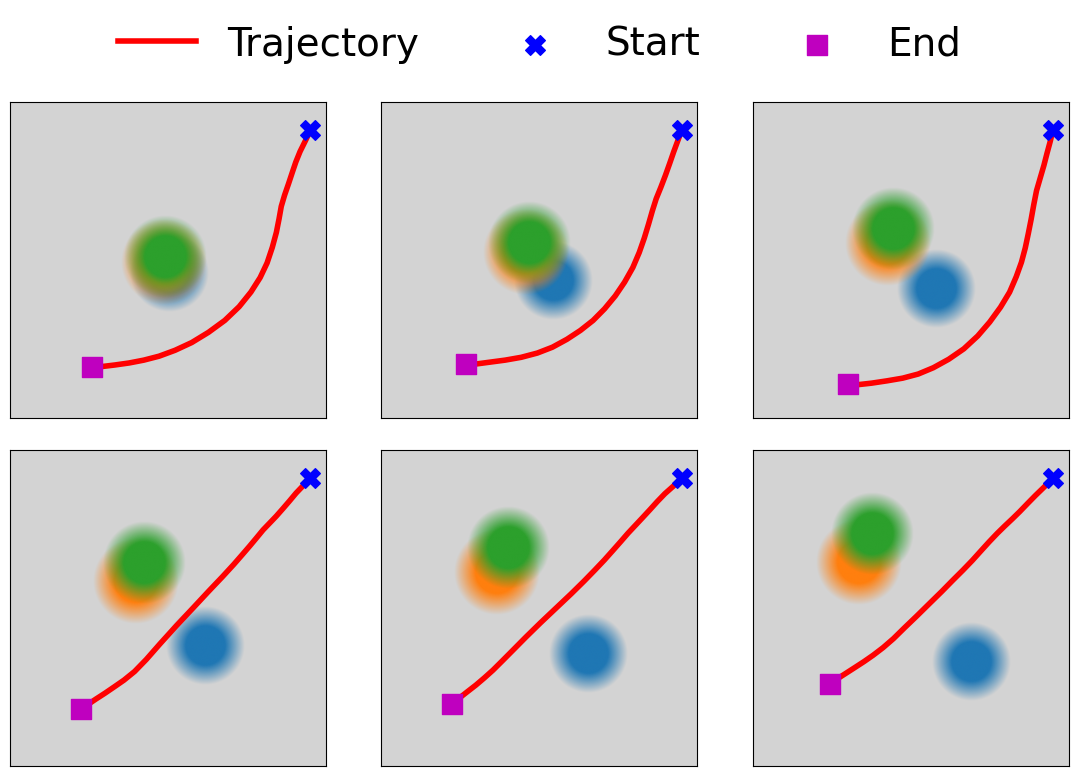}
    \caption{2D examples illustrating how the model adjusts its control strategy when Gaussian-bump obstacles are aggregated versus separated. Each bump is visualized with radius equal to 1.5 times its standard deviation and unit height. Predicted trajectories are shown in red, starting from the blue cross and ending at the terminal state (purple square) at $T=3$.}
    \label{fig:2d_stationary_2}
\end{figure}
% \begin{figure}[htbp]
%   \centering
%   % Row 1
%   \begin{subfigure}[t]{0.45\linewidth}
%     \centering
%     \includegraphics[width=\linewidth]{2D_stationary_bp2_1.png}
%     \caption{2D example with 2 stationary bumps.}
%     \label{fig:2d_stationary_bp2}
%   \end{subfigure}
%   \hfill
%   \begin{subfigure}[t]{0.45\linewidth}
%     \centering
%     \includegraphics[width=\linewidth]{2D_stationary_bp3_1.png}
%     \caption{2D example with 3 stationary bumps.}
%     \label{fig:2d_stationary_bp3}
%   \end{subfigure}

%   \vspace{6pt}

%   % Row 2
%   \begin{subfigure}[t]{0.45\linewidth}
%     \centering
%     \includegraphics[width=\linewidth]{2D_stationary_bp5_1.png}
%     \caption{2D example with 5 stationary bumps.}
%     \label{fig:2d_stationary_bp5}
%   \end{subfigure}
%   \hfill
%   \begin{subfigure}[t]{0.45\linewidth}
%     \centering
%     \includegraphics[width=0.5\linewidth]{figures/slaws.png}
%     \caption{Scaling laws across $N_b$ and $N$.}
%     \label{fig:scaling}
%   \end{subfigure}

%   \caption{2D stationary-bump experiments and scaling behavior. 
%   (a–c) Trajectories for 2, 3, and 5 stationary bumps. 
%   (d) Scaling laws: models trained with various $N_b$ and dataset sizes $N$. 
%   The plot shows mean relative MSE over $M=50$ test sets (reference solutions obtained by separately training on the test set).}
%   \label{fig:2d_nb_scaling}
% \end{figure}

\subsubsection{2D Moving Gaussian Bumps}
In this part, we set $\mX = [0.9,1]^2$. For each $j=1,\dots,N_b$, the bump centers are initialized as $\mathbf{c}_j \sim U([0.3,0.6]^2)$, the standard deviations are sampled as $\sigma_j \sim U[0.05,0.1]$, and the bumps move with velocities $\mathbf{v}_j \sim U([-0.15,0.15]^2)$. {In Figure~\ref{fig:2d_moving_bp2}, we compare the neural operator prediction, which leverages the bumps’ full future trajectories, with the \emph{instantaneous frozen-bump predictions}. At each snapshot time $t$, the \emph{instantaneous frozen-bump prediction} refers to the trajectory predicted under the assumption that the bumps are frozen at their current positions and solving the optimal control problem \eqref{eqn:Loss_gauss} coupled with \eqref{eqn:dynamics_gauss} over the remaining time horizon. The key message is that when the neural operator is provided with the full dynamics of the bumps, it is able to anticipate future configurations and make correct decisions. Even in situations where no passage exists at the current time, the model can foresee that an opening will appear in the near future and adjust its trajectory accordingly. For example, in both the left (two moving bumps) and right (three moving bumps) panels in Figure~\ref{fig:2d_moving_bp2}, at time $t=0.4$ and $t=0.9$ the bumps are clustered, so the instantaneous frozen-bump prediction, which is blind to future motion, takes conservative detours that are locally safe yet globally suboptimal. By $t=1.4$ and later, as the opening appear, \emph{instantaneous frozen-bump prediction} selects a feasible whose terminal state subsequently converges to the terminal state provided by neural operator.}

\begin{figure}[htbp]
    % \centering \includegraphics[width=0.95\linewidth]{figures/2D_two_trajs_bp1.png}
    % \includegraphics[width=0.95\linewidth]{figures/2D_two_trajs_bp2.png}
    \includegraphics[width=0.45\linewidth]{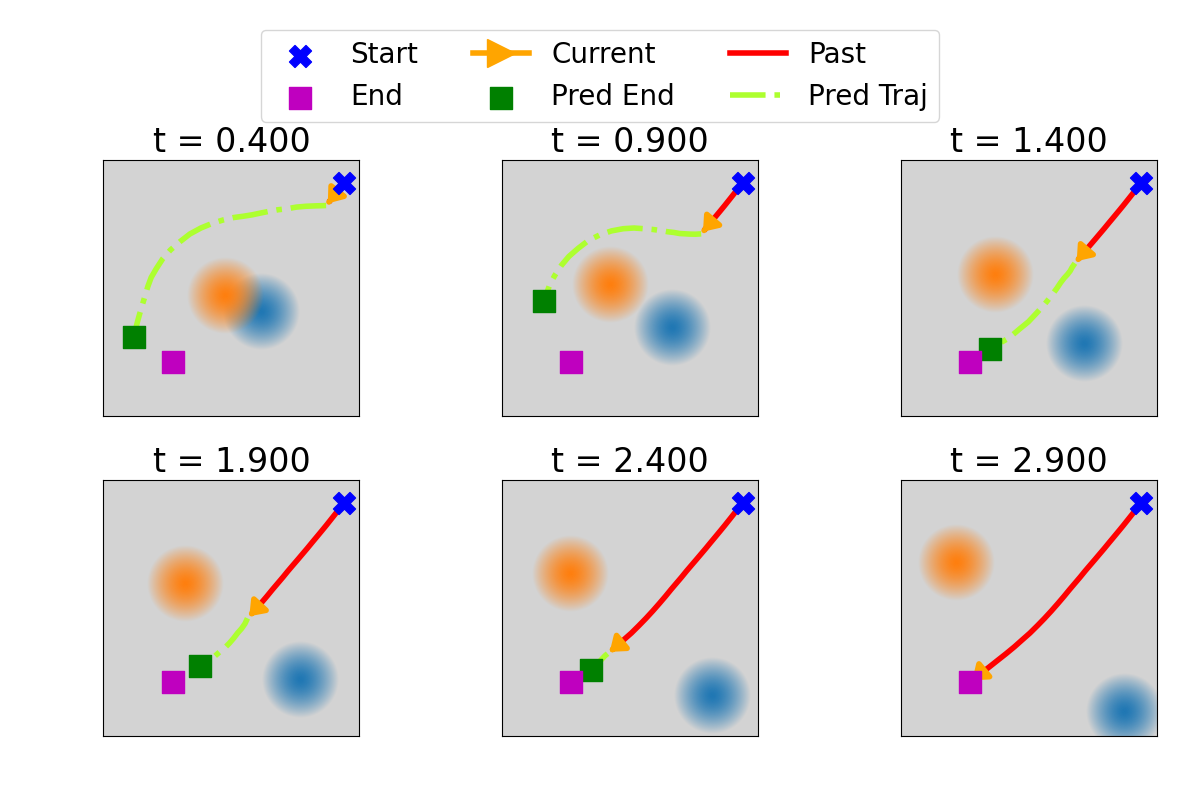}
    \includegraphics[width=0.45\linewidth]{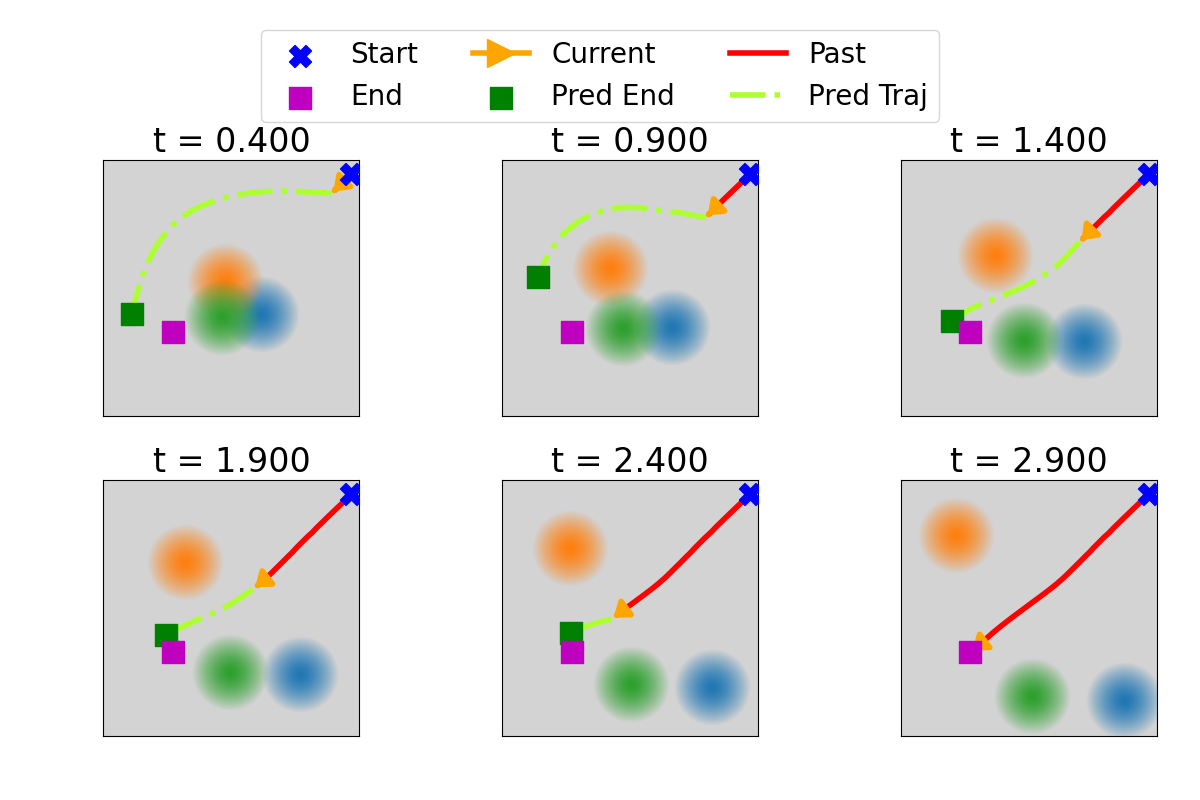}
    
    \caption{2D examples of moving Gaussian-bump environments with two bumps (left) and three bumps (right). Snapshots showing instantaneous frozen-bump prediction at different $t$. In each subfigure, the start is marked by a blue cross; the past trajectory up predicted by the Neural Operator to time $t$ is shown in red; and the current state at time $t$, with its instantaneous heading, is indicated by an orange arrow. The global-optimal terminal state predicted by the Neural Operator is shown as a purple square. The \emph{instantaneous frozen-bump prediction} is obtained by freezing the bumps at their locations at time $t$ and solving \eqref{eqn:Loss_gauss} that coupled with \eqref{eqn:dynamics_gauss} over the remaining horizon $[t, 3]$; the result is plotted as a light-green dashed curve and terminates at the predicted end position that marked by a green square.}
    
    \label{fig:2d_moving_bp2}
\end{figure}

\subsubsection{3D Gaussian Bumps}
In this part, we consider 3D setting with both stationary and moving bumps on the domain $\mX = [0.9,1]^3$. In stationary case, for each $j=1,\dots,N_b$, the bump centers are initialized as $\mathbf{c}_j \sim U([0.3,0.6]^3)$, and the standard deviations are sampled as $\sigma_j \sim U[0.05,0.1]$. Examples with $N_b=2, 3$ are presented in Figure~\ref{fig:3d_stationary_bp2}. For moving case, the bump centers are initialized as $\mathbf{c}_j \sim U([0.3,0.6]^3)$, the standard deviations are sampled as $\sigma_j \sim U[0.05,0.1]$, and the bumps move with velocities $\mathbf{v}_j \sim U([-0.15,0.15]^3)$. Numerical examples with $N_b=2$ and $N_b=3$ are shown in Figure~\ref{fig:3d_moving_bp2} and Figure~\ref{fig:3d_moving_bp3} respectively.
\begin{figure}[htbp]
    \centering \includegraphics[width=0.42\linewidth]{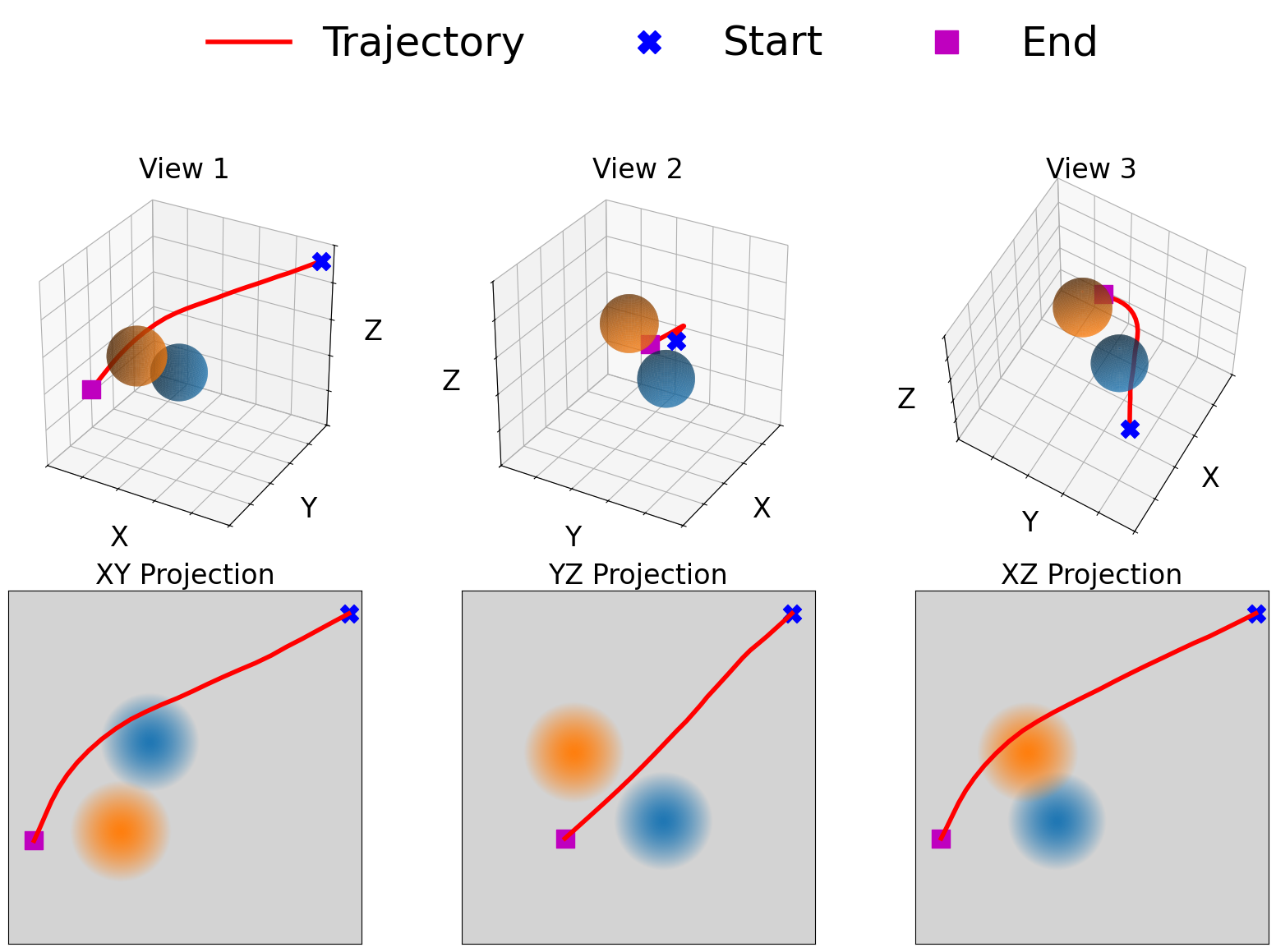}
    \hspace{8mm}
    \includegraphics[width=0.42\linewidth]{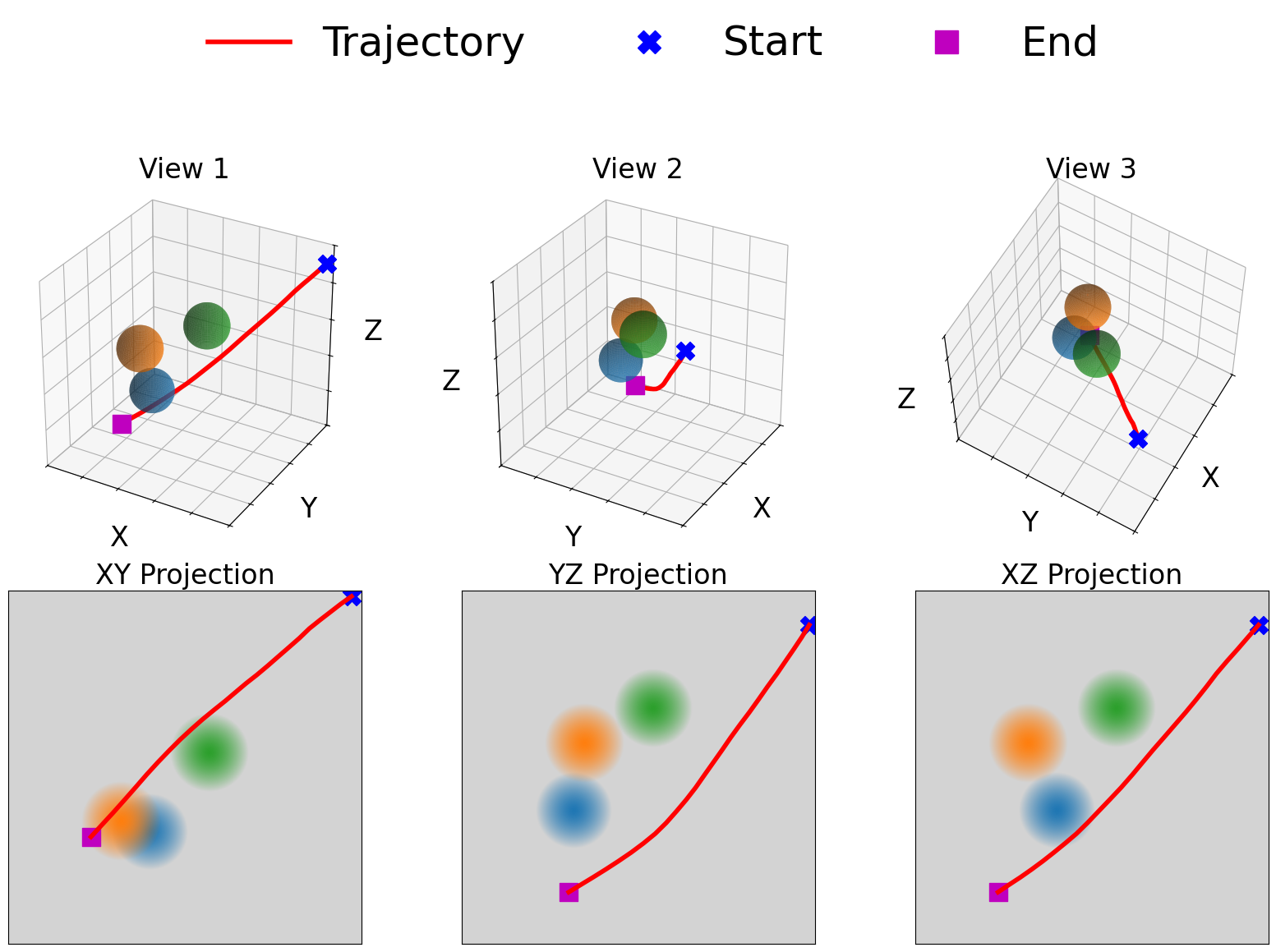}
    \caption{3D examples of trajectories in environments with stationary Gaussian bump obstacles. The left panel shows an instance with two randomly placed bumps, while the right panel shows an instance with three. The top row presents trajectories from different viewing angles, and the bottom row shows their projections onto the XY, YZ, and XZ planes. Each bump is visualized with radius equal to 1.5 times its standard deviation (both 3D views and 2D projections) and unit height. Red curves represent the predicted trajectories from the random starting point (blue cross) to the final state at $T=3$ (purple square).}
    \label{fig:3d_stationary_bp2}
\end{figure}

\begin{figure}[htbp]
    \centering \includegraphics[width=0.85\linewidth]{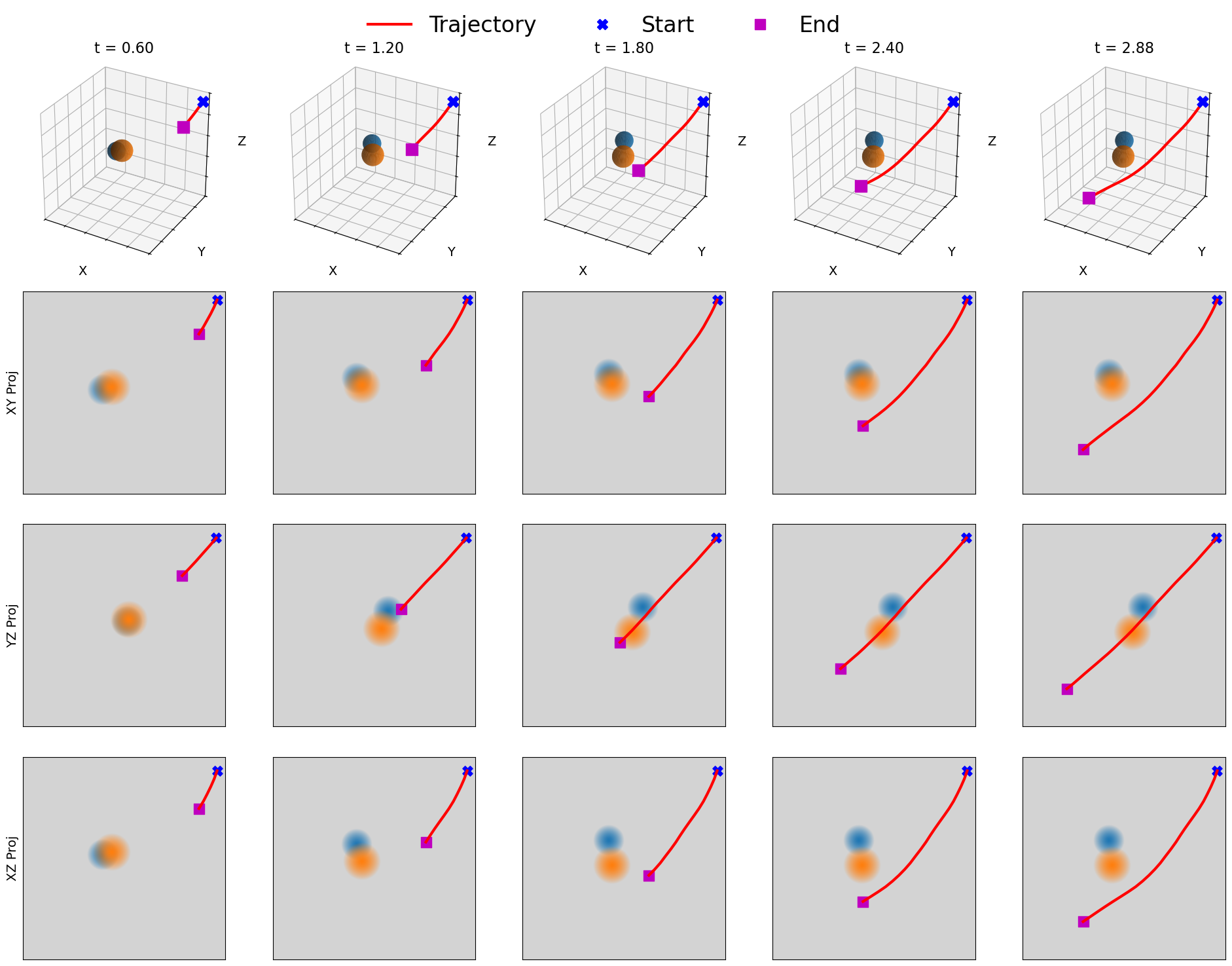}
    \caption{3D examples of trajectories in environments with \textit{two} moving Gaussian-bump obstacles. Each column corresponds to a time snapshot at $t$. The top row shows the trajectory (red) evolving over time, with the start marked by a blue cross and the current state at time $t$ by a purple square. The second to fourth rows show the corresponding projections onto the XY, YZ, and XZ planes. Each bump is visualized with radius equal to 1.5 times its standard deviation (applied to both 3D views and 2D projections) and unit height. These views illustrate how the trajectory adapts dynamically as the bumps move over time.}
    \label{fig:3d_moving_bp2}
\end{figure}

\begin{figure}[htbp]
    \includegraphics[width=0.85\linewidth]{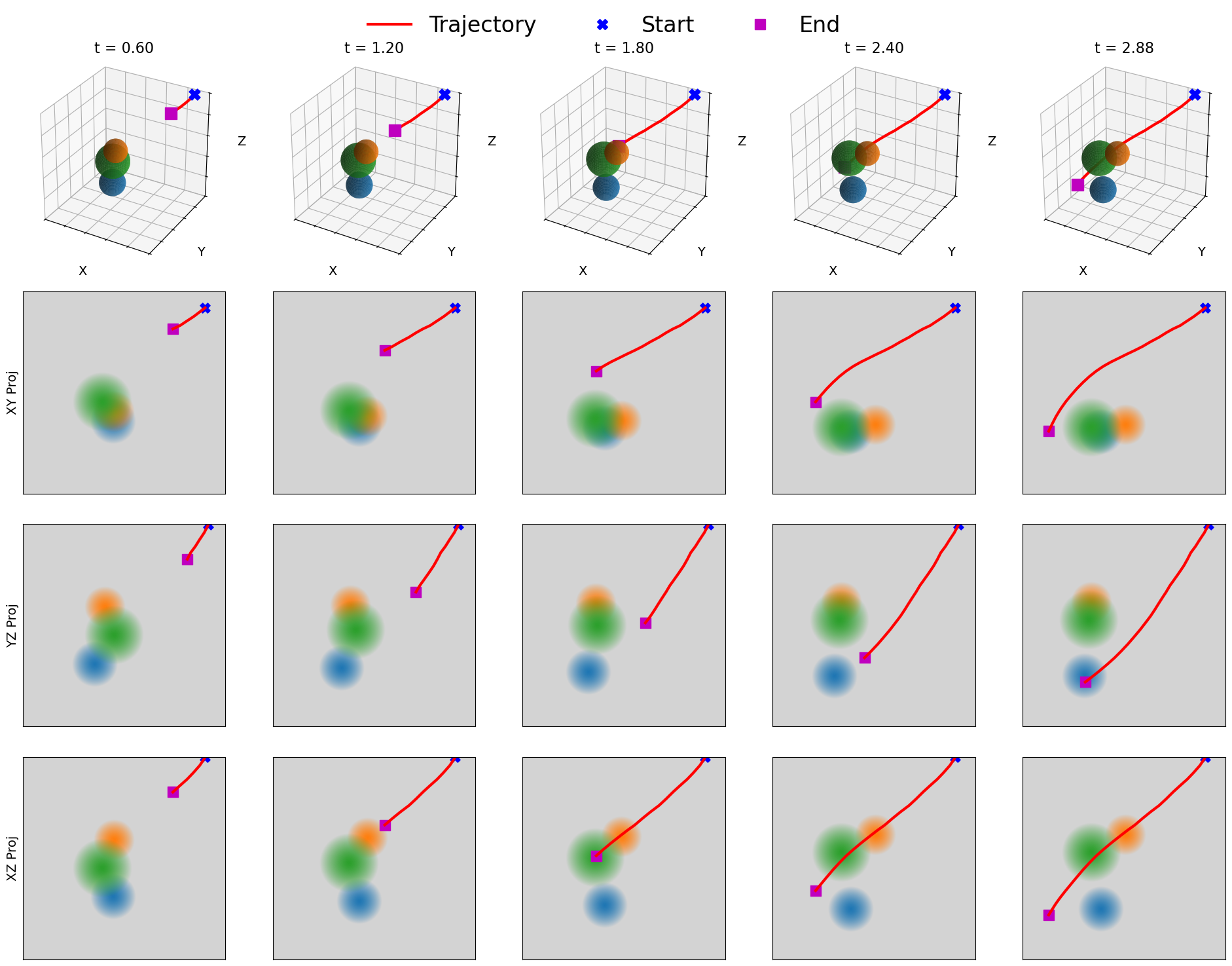}
    \caption{3D examples of trajectories in environments with \textit{three} moving Gaussian-bump obstacles. Each column corresponds to a time snapshot at $t$. The top row shows the trajectory (red) evolving over time, with the start marked by a blue cross and the current state at time $t$ by a purple square. The second to fourth rows show the corresponding projections onto the XY, YZ, and XZ planes. Each bump is visualized with radius equal to 1.5 times its standard deviation (applied to both 3D views and 2D projections) and unit height.}
    \label{fig:3d_moving_bp3}
\end{figure}

\subsection{Maze Navigation}\label{sec:maze}
In this subsection, we numerically evaluate the performance of our method on 2D maze navigation problems. We adopt a weighted variant of the control objective in equations~\eqref{eqn:Loss}, where distinct terms in the loss are assigned separate weights to balance their relative contributions. The resulting optimal control is given by:
\begin{equation}\label{eqn:Loss_maze}
     \bu^* \in  \arg \min_{\mathbf{u}} J_B(\mathbf{x}(t), \mathbf{u}) = w_r \int_0^T \frac{1}{2}|\bu|^2 dt +  w_B \int_0^T B(\bx(t))dt + w_T |\bx(T)|^2,
\end{equation}
subject to 
\begin{equation}
\left\{
\begin{aligned}
    \dot{\mathbf{x}}(t) &=  \mathbf{u}(t, \bx(t)), \quad t \in [0, T] \\
    \mathbf{x}(0) &= \mathbf{x}_0 .
\end{aligned}
\right.
\end{equation}
we fix the initial condition as $\bx_0 = (1,1)$ and enforce the terminal position to be the origin by assigning a large value to the terminal cost weight $w_T$. The barrier function is defined by the maze construction procedure described in Algorithm~\ref{alg:maze} (Appendix~\ref{sec:network_maze}). In the following numerical experiments, we fix the loop coefficient at 0.4 and set the maze width equal to its height (e.g., a size $5 \times 5$ maze has width 5 and height 5). To encourage the trajectory to avoid obstacles, we assign a large weight $w_B$ to the barrier loss term. The maze barrier is constructed using a set of axis-aligned rectangular barriers; see Appendix~\ref{sec:network_maze} for details. Each rectangle is specified by a tuple $[a, b, c, d]$, where $a$ and $b$ denote the left and right x-coordinates, and $c$ and $d$ denote the bottom and top y-coordinates of the rectangle, respectively. The barrier loss is then be computed by 
\begin{equation}
    B(\bx) = \sum_{j=1}^{N_b} B_j(\bx; a^j, b^j, c^j, d^j)
\end{equation}
where the $N_b$ denote number of rectangles and each $B_j$ represents the barrier loss of the $j$-th rectangular obstacle:
\begin{equation}\label{eqn:wall}
    B_j(\bx; a^j, b^j, c^j, d^j)) = \sigma(\lambda(x-a^j))*\sigma(\lambda(b^j-x))*\sigma(\lambda(y-c^j))*\sigma(\lambda(d^j-y)),
\end{equation}
where $\sigma$ denotes the sigmoid function (not to be confused with the standard deviation in the Gaussian mixture example) and $\lambda$ is the sharpness parameter. 

In this experiment, we consider three levels of problem complexity: mazes generated by Algorithm~\ref{alg:maze} with \(w=h=5\) (denoted \(5\times5\), easiest), \(w=h=7\) (denoted \(7\times7\), moderate), and \(w=h=9\) (denoted \(9\times9\), hardest). For each level, we use the same training set size and evaluate on an identically sized test set. We present ten post-training inference instances: \ref{fig:maze55} for \(5\times5\), \ref{fig:maze77} for \(7\times7\), and \ref{fig:maze99} for \(9\times9\). As maze complexity increases, trajectory quality degrades: in \(5\times5\), all trajectories are successful and remain within free space; in \(7\times7\), several paths graze edges or corners; and in \(9\times9\), many fail, with trajectories crossing walls. This trend is reflected in the boxplots of obstacle loss across maze sizes in Figure~\ref{fig:boxplot}, which increase with complexity.

% We consider three levels of problem complexity: mazes generated by Algorithm~\ref{alg:maze} with $w=h=5$ (denoted $5\times 5$, easiest), $w=h=7$ (denoted $7\times 7$, moderate) and hardest one $w=h=9$ (denoted $9 \times 9$, hardest). For each level, we use the same training set size and evaluate on an identically sized test set. 

% We presented 10 inference instances: \ref{fig:maze55} for $5 \times 5$, \ref{fig:maze77} for $7 \times 7$, and \ref{fig:maze99} for $9 \times 9$.

% And we present boxplot of the resents boxplots of the obstacle loss across different maze
% sizes in \ref{fig:boxplot}, illustrating how trajectory feasibility degrades as maze complexity increases. As maze complexity increases, trajectory quality degrades: in the case of $5\times 5$, all trajectories are successful and remain within free space; in 
% $7\times 7$ several paths graze edges or corners; and in $9\times 9$ many fail, with trajectories crossing walls. 

% This trend is consistent with the obstacle-loss statistics in panel (d), which grow with complexity.

\begin{figure}[t]
  \centering
  \captionsetup[subfigure]{justification=centering, font=small}
  % --- Row 1 ---
  \begin{subfigure}[t]{0.48\textwidth}
    \includegraphics[width=\linewidth]{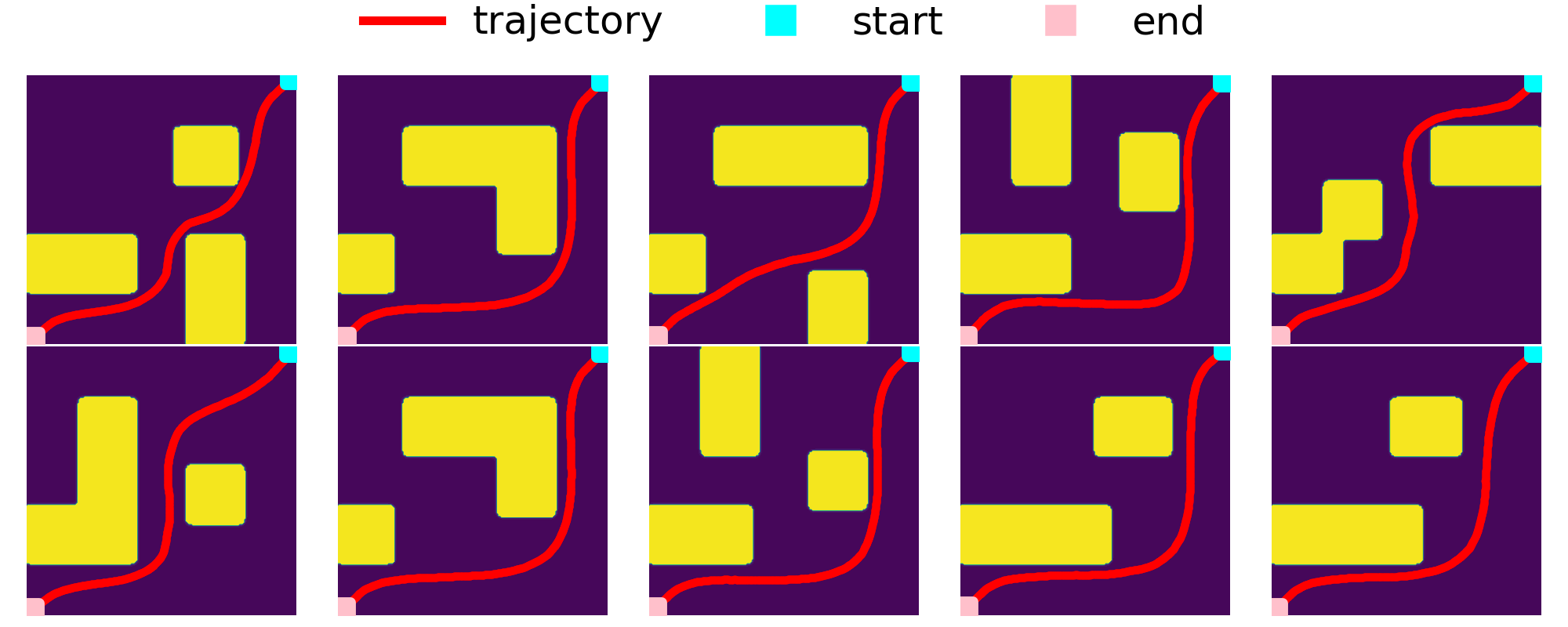}
    \caption{5×5}
    \label{fig:maze55}
  \end{subfigure}\hfill
  \begin{subfigure}[t]{0.48\textwidth}
    \includegraphics[width=\linewidth]{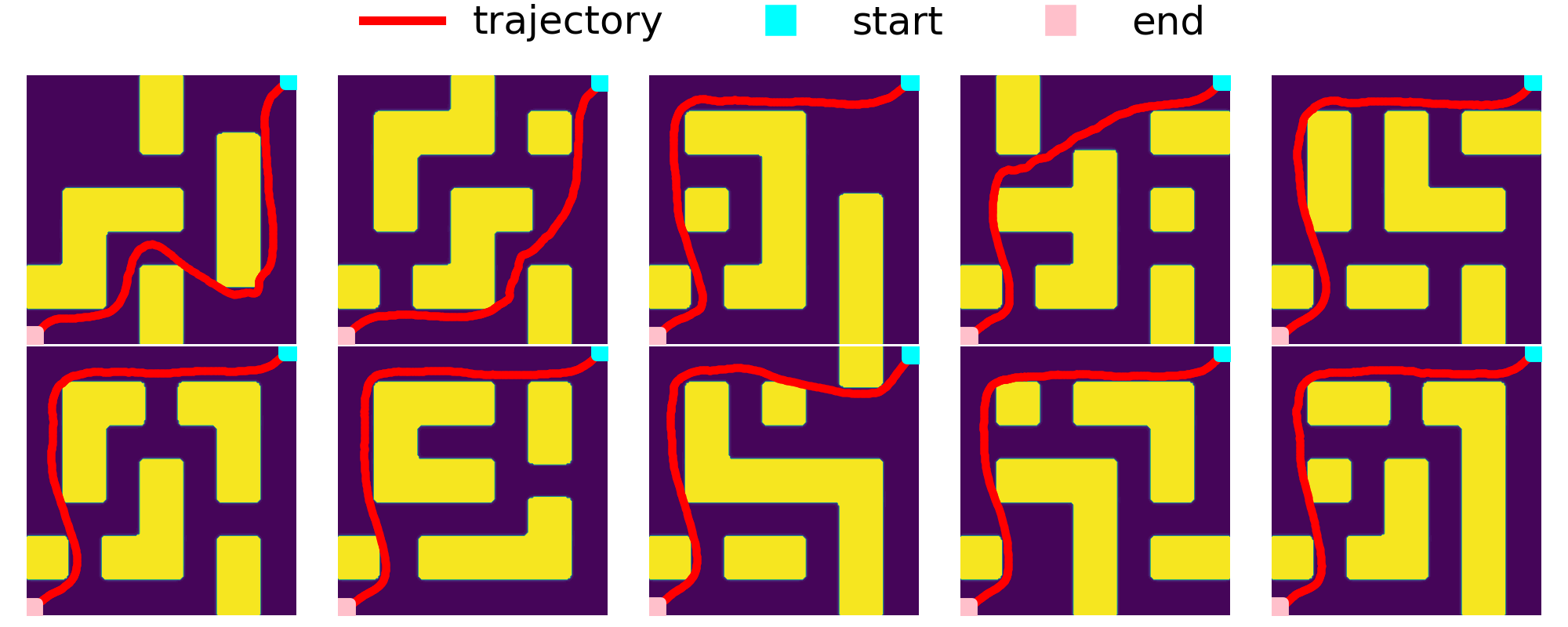}
    \caption{7×7}
    \label{fig:maze77}
  \end{subfigure}

  \vspace{0.7em} % <-- vertical gap between rows (tune)

  % --- Row 2 ---
  \begin{subfigure}[t]{0.48\textwidth}
    \includegraphics[width=\linewidth]{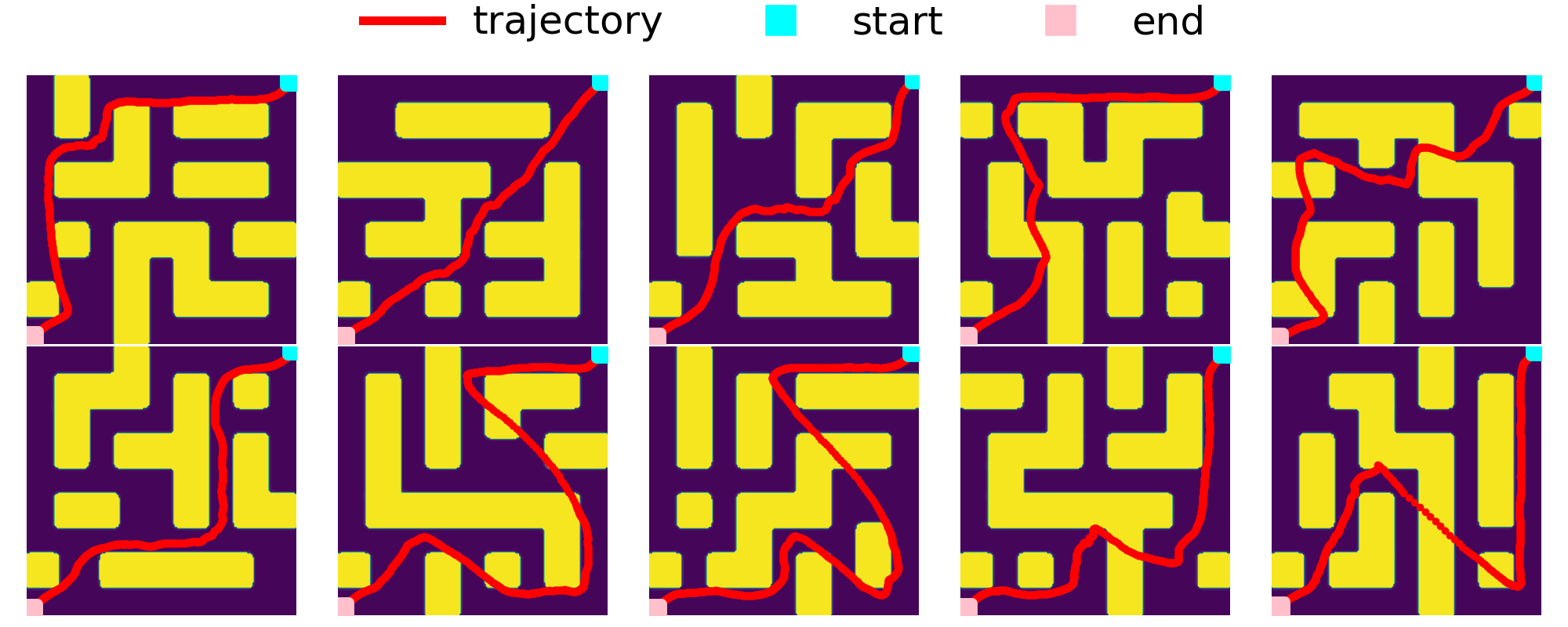}
    \caption{9×9}
    \label{fig:maze99}
  \end{subfigure}\hfill
  \begin{subfigure}[t]{0.48\textwidth}
    \includegraphics[width=\linewidth]{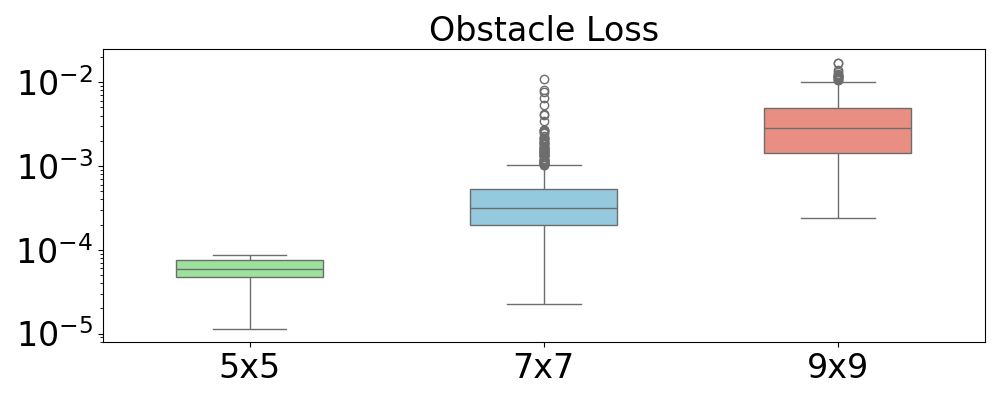}
    \caption{Obstacle loss}
    \label{fig:boxplot}
  \end{subfigure}
  \caption{Trajectory prediction in maze environments of increasing complexity and the corresponding obstacle loss. Panels (a)–(c) show 10 representative trajectories (red) from the start (cyan square) to the end (pink square) in mazes of size $5\times 5$, $7\times 7$, and $9\times 9$, respectively. Panel (d) presents boxplots of the obstacle loss across different maze sizes, illustrating how trajectory feasibility degrades as maze complexity increases.}
  \label{fig:maze}
\end{figure}

\subsection{Unicycle}\label{sec:unicycle}
In this subsection, we extend our study to the unicycle problem in a 2D physical space with nonlinear dynamics. We adopt the enhanced unicycle model described in Section 13.2.4 of \cite{lavalle2006planning}. The corresponding optimal control problem is formulated as follows:
\begin{equation}\label{eqn:car}
     \bu^*  \in  \arg \min_{\bu} J_B(\mathbf{x}(t), \bu) = w_r \int_0^T \frac{1}{2} |\bu|^2 \rd t +  w_B \int_0^T B(\bx(t)) \rd t + w_t |\bx(T)-\bx_T|^2
\end{equation}
where the state is given by $\bx(t)=(x(t), y(t), \phi(t), v(t))$. Here $\phi$ denotes the heading angle of the vehicle with respect to the x-axis, and $v$ denotes its velocity. The control function is $\bu(t)=(u(t), a(t))$, where $u(t)$ denotes the angular velocity (i.e. steering rate) and $a(t)$ denotes the acceleration. The target terminal state is denoted by $\bx_T = (x_T, y_T, \phi_T, v_T)$, and $w_r$, $w_B$, and $w_T$ are positive weighting parameters. The dynamics is given by:
\begin{equation}
\begin{cases}
\dot{x}(t) = v(t) \cos\phi(t), \\
\dot{y}(t) = v(t) \sin\phi(t), \\
\dot{\phi}(t) = u(t), \\
\dot{v}(t) = a(t), \\
(x(0),y(0),\phi(0),v(0)) = (x_0,y_0,\phi_0,v_0).
\end{cases}
\end{equation}
In this setting, we assume that obstacles are parametrized by gaussian bumps defined in \eqref{eqn:gauss}, which is 
\[
B(\bx(t)) = \sum_{j=1}^{N_b} \mN(\bx; \mathbf{c}_j, \sigma_j, h_j).
\]
In the following examples, we set $w_r = 1$, $w_B=50$, and $w_t=200$, and fix the initial condition as 
\[
(x_0,y_0,\phi_0,v_0) = (0, 0, \tfrac{\pi}{2}, 0),
\]
which corresponds to a vehicle starting at the origin, facing upward, and initially at rest. In our experiments, we randomly sample the terminal position and heading angle while fixing the terminal velocity to zero. Specifically, we take $(x_T, y_T) \sim U([0.9,1]^2)$, $\phi_T \sim U[-\pi,\pi]$, and set $v_T = 0$. 

In Figure~\ref{fig:car}, we compare our neural operator method with a  baseline nonlinear-programming (NLP) solver \cite{kelly2017introduction,Kelly_OptimTraj_Trajectory_Optimization_2022}.In the NLP formulation, the nonlinear dynamics and complex constraints are discretized and solved directly, without relying on the necessary conditions of optimality. While widely used, NLP offers no convergence guarantees and is sensitive to the initial guess. In this experiment, we set initial guess to all instances as $u_{guess}=0$ and $a_{guess}=0$. As demonstrated, in some cases our neural operator approach provides control similar to the NLP baseline, though not always. We also analyze the statistics of the cost gap between the two methods (Figure~\ref{fig:car}, right panel) and find that our approach achieves even lower cost in about 15 of the 100 test tasks, while most results are centered around zero—indicating that the neural operator produces controls that are close to, but slightly suboptimal compared to NLP.

Nevertheless, a key advantage of our method is its efficiency: inference with the neural operator is nearly instantaneous. We conducted 100 test tasks using both NLP and NO on an Apple M3 Pro chip. The average computation time of NLP was 4.42 s, whereas the average inference time of NO was only 0.00184 s

% \begin{figure}[htbp]
%     \centering \includegraphics[width=0.8\linewidth]{figures/car_obs1.png}
%     \caption{5 examples of unicycle control problem.}
% \end{figure}

\begin{figure}[htbp]
    \centering \includegraphics[width=0.6\linewidth]{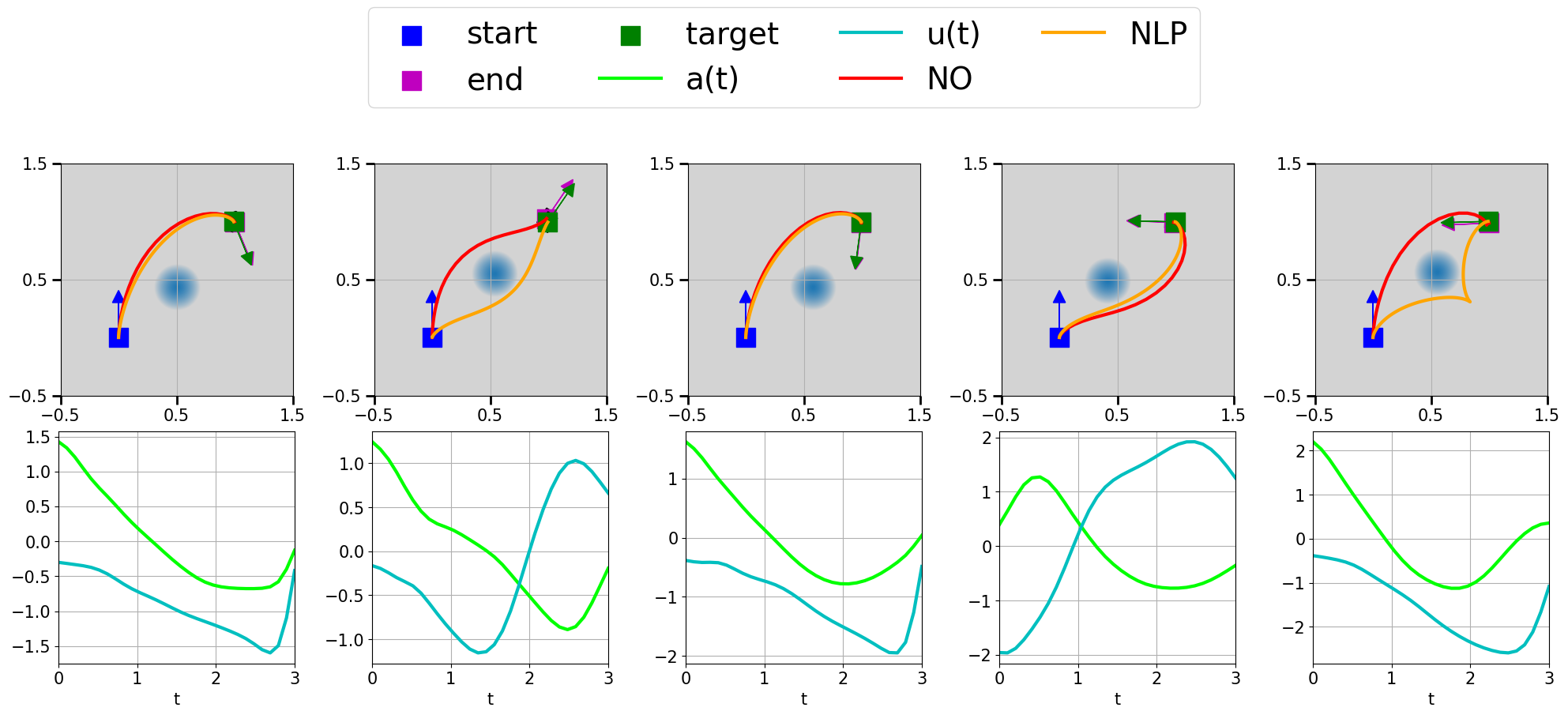}
    \includegraphics[width=0.3\linewidth]{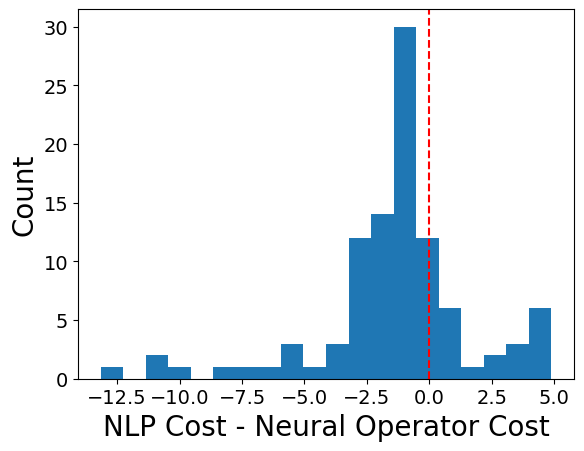}
    \caption{Examples of unicycle car control. The left panel compares neural operator (NO)–based control with nonlinear programming (NLP) solutions across five instances. The top row shows trajectories from the start state (blue square with arrow) to the target state (green square with arrow) and the final state (purple square), with NO predictions in red and NLP predictions in orange. The bottom row plots the corresponding control signals predicted by the NO, with $u(t)$ in cyan and $a(t)$ in green. The histogram on the right shows the distribution of cost differences (NLP cost minus NO cost) across test cases, where positive values indicate that the neural operator achieves lower cost.}
    \label{fig:car}
\end{figure}

\subsection{Closed-Loop Control of the Unicycle Robot via Model Predictive Control}\label{sec:MPC}
In this subsection, we demonstrate the effectiveness of the extended solution operator combined with MPC, introduced in Section~\ref{sec:extension}, yields \emph{closed-loop} optimal control predictions for on a unicycle robot delivery task. The robot moves in a maze-like environment with unknown obstacles representing people that may appear unexpectedly along its path. Meanwhile the robot must navigate from a prescribed initial position to a target location using only locally available obstacle information. Specifically, the robot has full access to the static maze layout, which may serve as a known map, while the pedestrians are modeled as dynamic obstacles that are unknown a priori and can only be detected when they enter the robot’s sensing region. Precisely, we formulate this unicycle delivery task as the following free-terminal-time optimal control problem:
\begin{equation}\label{eqn:free_time_loss}
{ \mathbf{u}^*(t)}, { T^*}  \in  \arg \min_{\bu, T} J_B(\mathbf{x}(t), \bu) = w_1 { \int_0^T |\mathbf{u}(t)|^2 \mathrm{d} t } + w_2 {  \int_{0}^ { T} B(\bx(t))  \mathrm{d} t} + w_3  \|\mathbf{x}({ T}) - { \mathbf{x}_T} \| + w_4 { T}
\end{equation}
\begin{equation}\label{eqn:uni_dyn}
\text{s.t.}\quad
\begin{cases}
\dot{x}(t) = v(t)\cos \phi(t), \\
\dot{y}(t) = v(t)\sin \phi(t), \\
\dot{\phi}(t) = u(t), \\[4pt]
\dot{v}(t) = a(t), \\[6pt]
{(x(0),\, y(0),\, \phi(0),\, v(0)) = (x_0,\, y_0,\, \phi_0,\, v_0):=\mathbf{x}_0}.
\end{cases}  
\end{equation}
Here the function $B$ represents the aggregated obstacles:
\begin{equation}
    B(\bx) = \sum_{j=1}^{N_{b1}} B_j(\bx; a^j, b^j, c^j, d^j) + \sum_{l=1}^{N_{b2}} \mN_l(\bx; \bc^l, \sigma^l)
\end{equation}
which consists of two components. The terms $B_j$ encode the maze walls and follow the definition in~\eqref{eqn:wall}, while the functions $\mN_l$ are parameterized Gaussian bumps used to approximate static unknown obstacles (e.g., stationary pedestrians). To model the robot’s limited sensing capability, we assume that such unknown obstacles become visible only when they lie within a prescribed vision radius $R>0$. Given a training dataset $\Gamma$, we then train a free-terminal-time solution operator using Algorithm~\ref{alg:training}, where the cost functional is adapted from~\eqref{eqn:Loss} to~\eqref{eqn:free_time_loss} together with the dynamics~\eqref{eqn:uni_dyn}. Once the learned operator $\mathcal{G}_{\Gamma;\tilde{\theta}}$ is obtained, we generate trajectories using the MPC procedure in Algorithm~\ref{alg:mpc_inference}, where at each step $k$ the obstacle field is updated as
\[
B_k(\mathbf{x}) 
  = \sum_{j=1}^{N_{b1}} B_j(\mathbf{x}; a^j, b^j, c^j, d^j)
    + \sum_{l \in S_k} \mathcal{N}_l(\mathbf{x}; \mathbf{c}^l, \sigma^l),
\]
where the index set of currently visible unknown obstacles is
\[
    S_k = \big\{\, l \;\big|\; 
    \|\mathbf{c}^l - \mathbf{x}(k\Delta t)\|_\infty \le R \,\big\}.
\]

We remark that the above formulation is general and can accommodate various maze configurations, since the maze walls are treated as part of the input and the learned operator can adapt accordingly. The downside is that this flexibility requires a substantially larger training dataset and more computational resources. For simplicity, in the following experiments we fix the maze configuration by using three static walls, while allowing only the unknown obstacles to vary. Figure~\ref{fig:MPC} presents two representative instances of the resulting navigation procedure, 
illustrating that the proposed solution operator delivers remarkably well trajectories in the inference stage.
Each instance visualizes the first 15 MPC updates, corresponding to time points $t = k \Delta t$ with $\Delta t = 0.04$ and $k = 1, \dots, 15$. The robot’s current pose is shown as a blue rectangle with a heading arrow, the executed trajectory appears as a solid blue curve, and the MPC prediction based on the current obstacle field is plotted as a dashed cyan curve. The sensing radius is indicated in light blue, static maze walls are shown in gray, and active Gaussian bumps corresponding to influential obstacles are highlighted in red. Once an obstacle is no longer influential, its bump is deactivated and rendered in pink. As the robot navigates in the maze, newly visible obstacles trigger immediate replanning, illustrating the robustness and adaptability of the MPC strategy under partial observability. The construction of the dataset $\Gamma$, additional training details, and further implementation choices for the MPC rollout are deferred to Section~\ref{sec:MPC_details}.

% Implementation details are provided in Section~\ref{sec:MPC_details}, and figure~\ref{fig:MPC} presents two representative instances of the resulting navigation procedure. Each instance shows the first 15 MPC updates, corresponding to time points $t = k \Delta t$ with $\Delta t=0.04$ and $k=1, \cdots 15$. The robot’s current pose is displayed as a blue rectangle with a heading arrow, the past executed trajectory appears in solid blue, and the MPC prediction assuming no unseen obstacles is drawn as a dashed cyan curve. The sensing radius is shown in light blue, static maze walls in gray, and activated Gaussian bumps in red. Once an obstacle is no longer influential, its bump becomes inactive and is shown in pink. As the robot traverses the maze, newly visible obstacles trigger immediate replanning, demonstrating the robustness and adaptability of the MPC strategy under partial observability.

\begin{figure}[htbp]
    \centering \includegraphics[width=0.88\linewidth]{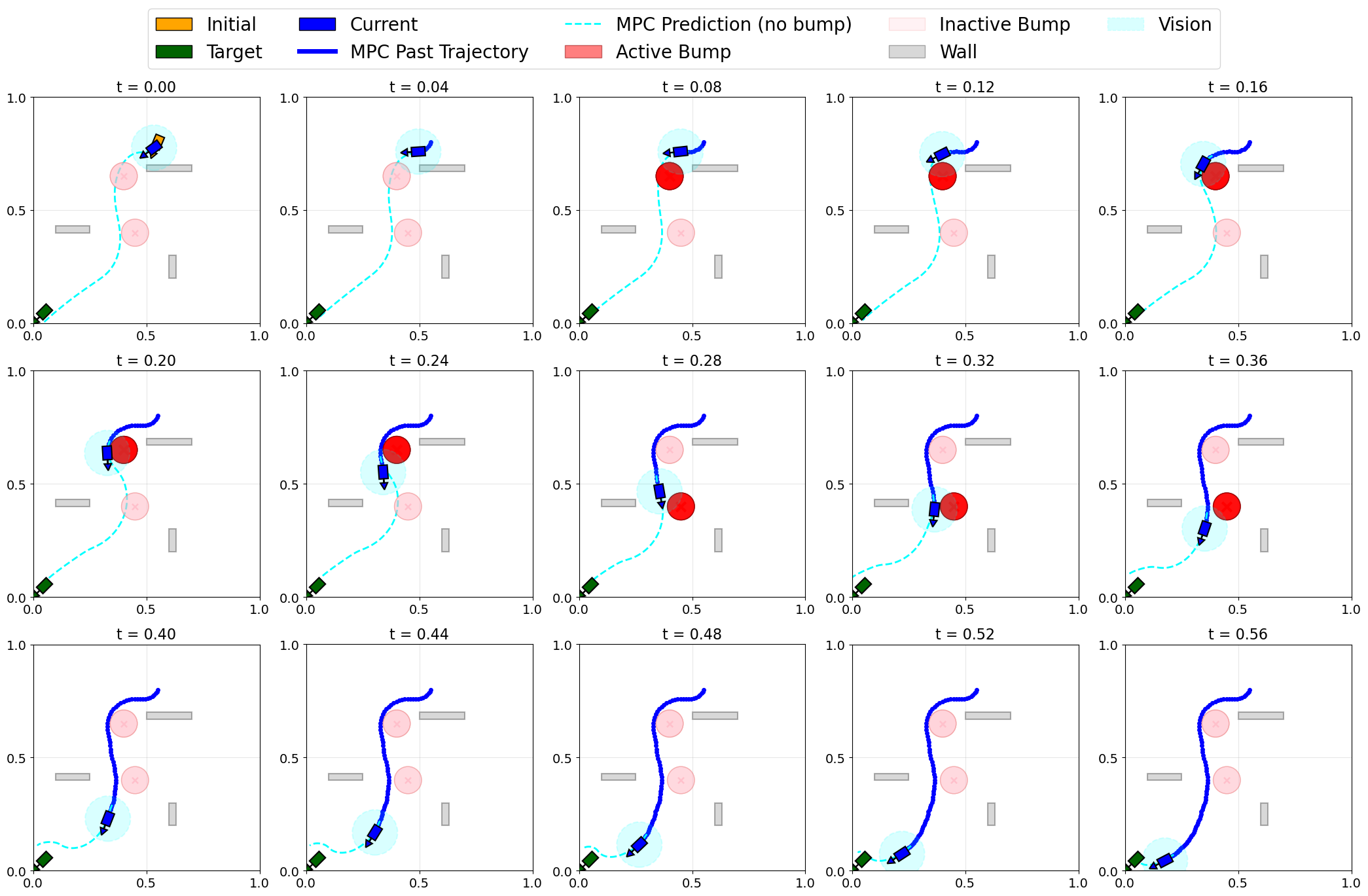}
    \includegraphics[width=0.88\linewidth]{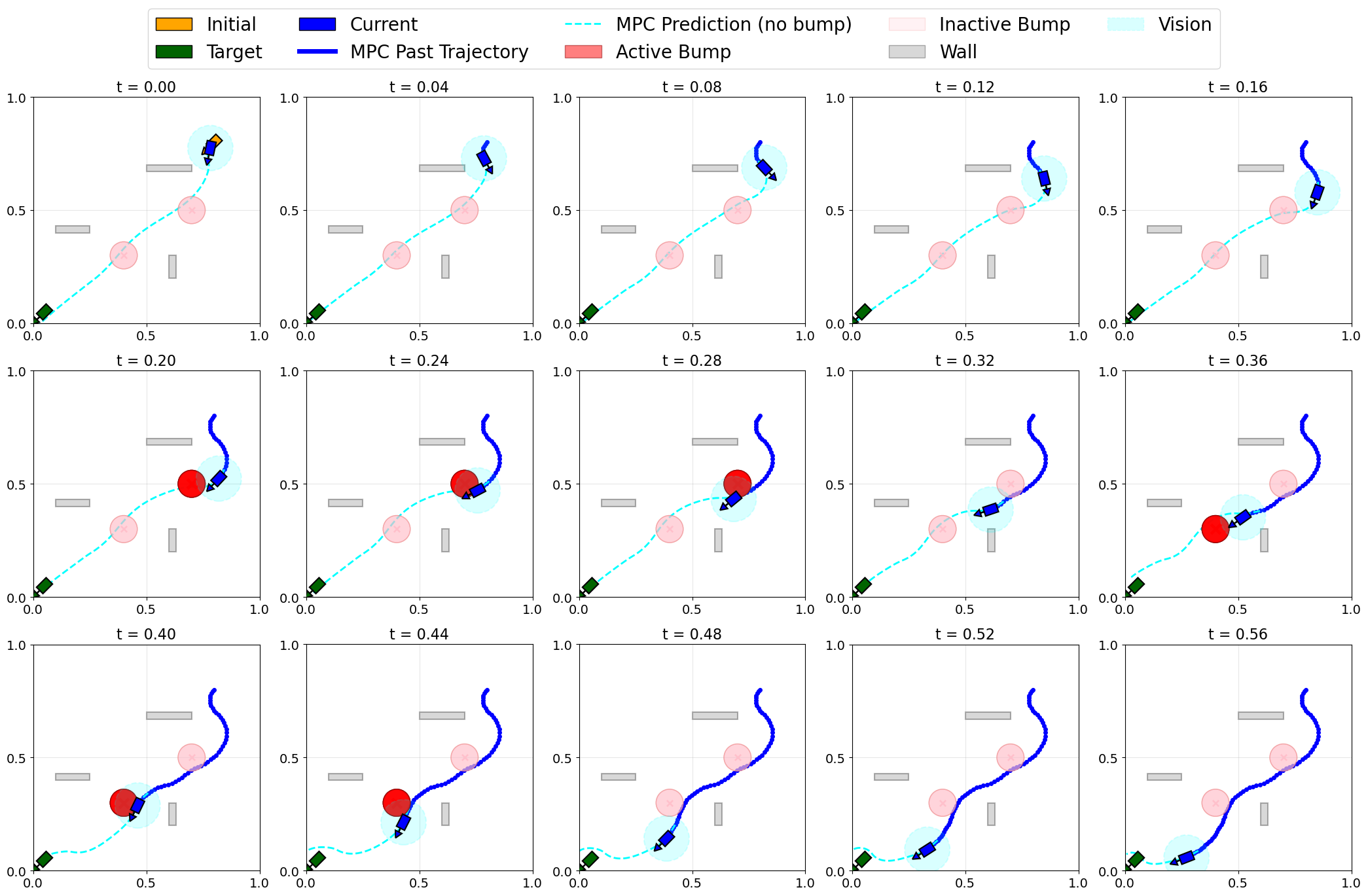}
    \caption{MPC-based navigation inside a maze with partially observable obstacles. The first three rows and the last three rows correspond to two independent instances. In each instance, every subplot shows the robot’s predicted and executed trajectories at time $t = k \Delta t$ with $\Delta t=0.04$. The robot (blue rectangle with heading arrow) moves from the initial state (orange) toward the target (green) while unknown obstacles—modeled as Gaussian bumps—become activated (red) only when entering the robot’s sensing radius (light blue). Gray rectangles denote maze walls, the solid blue curve shows the executed trajectory, and the dashed cyan curve shows the MPC prediction ignoring unseen Gaussian bumps.}
    \label{fig:MPC}
\end{figure}

\subsection{Scaling Laws}
In addition to evaluating the effectiveness and efficiency of our proposed method, we also investigate its scaling behavior. As indicated in Theorem~\ref{thm:gen}, increasing the problem complexity, measured by the intrinsic dimension, degrades performance when all other hyperparameters and training settings are fixed. 

To empirically verify this, we vary the intrinsic dimension in several problem and examine how the error scales with the number of training samples. We consider three representative examples: the linear–quadratic regulator (LQR), the 2D Gaussian-bump problem, and the 2D maze navigation problem. The setups of the latter two are described in Sections~\ref{sec:gauss_stationary} and~\ref{sec:maze}, respectively. For the LQR case, we use a scalar system with cost functional
\[
J(x(t), u(t)) = \int_0^T \left(q \left(x(t)\right)^2  + r \left(u(t)\right)^2  \right)\rd t + \lambda (x(T) - \xi)^2
\]
subject to the linear dynamics
\[
\dot{x}(t) = u(t), ~ x(0) = x_0.
\]
The corresponding analytic open-loop optimal control is
\[
u^*(t; q, r, \lambda, \xi, x_0) = \zeta(x_0 \sinh(\zeta t) + \eta \cosh(\zeta t))
\]
with $\zeta = \sqrt{\frac{r}{q}}$ and $\eta = \frac{y^* - x_0 \cosh(\zeta T)}{\sinh(\zeta T)}$, here $y^*$ is defined as follow:
\[
y^* = \begin{cases}
    m-B, \quad \xi < m-B, \\
    \xi, \quad m-B \leq \xi \leq m+B, \\
    m+B, \quad \xi \geq m+B.
\end{cases}
\]
with $m = \frac{x_0}{\cosh(\zeta T)}$ and $B = \frac{\lambda}{2 \sqrt{qr} \cosh(\zeta T) }$.
To test the scaling law, we consider four operator learning tasks with increasing input dimension:
\begin{enumerate}
    \item DOF=2: $\mG_2: (\xi, x_0) \mapsto u^*(t;\xi, x_0)$, with $q=1, r=1, \lambda = 1$, $\xi\sim U[-0.5,1.5]$ and $x_0 \sim U[0.5,1]$.
    \item DOF=3: $\mG_3: (r, \xi, x_0) \mapsto u^*(t;r, \xi, x_0)$, with $r=1, \lambda=1$, $\xi \sim U[-0.5,1.5]$ and $r, x_0 \sim U[0.5,1]$.
    \item DOF=4: $\mG_4: (q, r, \xi, x_0) \mapsto u^*(t;q, r, \xi,  x_0)$ with $\lambda=1$, $\xi \sim U[-0.5,1.5]$ and $q, r, x_0 \sim U[0.5,1]$. 
\end{enumerate}
In each setting, the degrees of freedom (DOF) is the number of independent scalar inputs that parametrize the optimal control function, i.e., the intrinsic dimension of the learning task. In Figure~\ref{fig:scaling_lqr}, we report the relative mean squared error (RMSE) between the neural-operator predictions and the analytical reference solutions (left) and the relative error between the predicted costs and the analytical optimal costs (right). As the number of degrees of freedom (DOF) increases, both errors decay more slowly as the number of training samples $N$ grows. This trend is consistent with our theoretical picture: due to the regime switching induced by the absolute-value terminal term, the optimal function $u^*(t; q,r,\lambda,\xi,x_0)$ is only $C^{0,1}$ in $\xi$. Accordingly, our main theorem\ref{thm:gen} predicts an asymptotic scaling of order $N^{-1/(1+DOF/2)}$ when $N$ is sufficiently large. In practice, however, several factors make this asymptotic regime difficult to observe in our experiments. First, the cost error can reach the $10^{-4}$ magnitude, at which point time-discretization error from the neural-ODE solver becomes non-negligible. Reducing this numerical error requires either much smaller time steps or higher-order integration schemes, both of which substantially increase training cost. Second, the optimization landscape is nonconvex and we train with Adam, so the attained solutions may be limited by optimization error rather than approximation or statistical error. Although we do not cleanly recover the predicted scaling exponent in these finite-compute experiments, we consistently observe slower error decay as DOF increases, supporting our “price-to-pay” argument.

As the complexity For the stationary Gaussian-bump problem, the intrinsic dimension corresponds to the number of bumps $N_b$ and in Figure~\ref{fig:scaling} (left), we report the relative mean squared error (RMSE) between neural operator predictions and reference solutions, where each reference solution is obtained by individually solving the minimization problems in \cref{eqn:Loss_gauss} and \cref{eqn:dynamics_gauss}. As $N_b$ increases, the RMSE decays more slowly with respect to the number of training samples $N$. This trend is clearly reflected by the dashed log–log fitting lines: the fitted slope becomes less negative (i.e., closer to zero) as $N_b$ grows, indicating a smaller decay exponent and thus reduced sample efficiency at higher intrinsic dimension. For the maze navigation problem, the intrinsic dimension reflects the structural complexity of the maze, typically captured by its size. In Figure~\ref{fig:scaling} (right), we report the obstacle loss defined in \cref{eqn:Loss_maze}, aggregated across mazes of different sizes. The decay rates appear sharper than in the Gaussian-bump case, which is due to our evaluation criterion: we only check whether the predicted trajectory successfully avoids obstacles, rather than comparing it against exact reference trajectories, which are typically unavailable.

\begin{figure}[htbp]
    \centering 
    \includegraphics[width=0.43\linewidth]{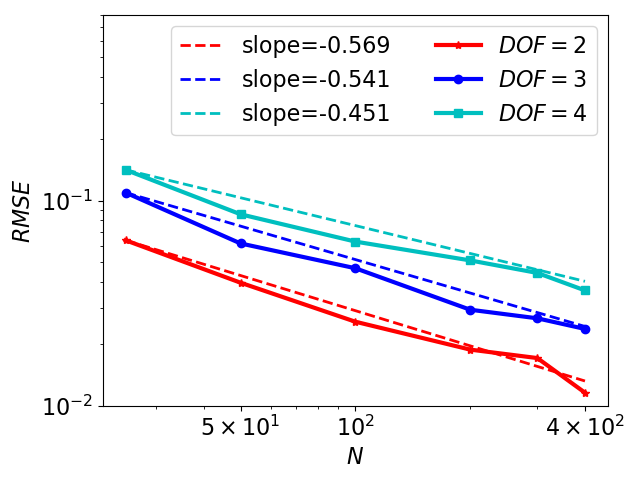}
    \includegraphics[width=0.43\linewidth]{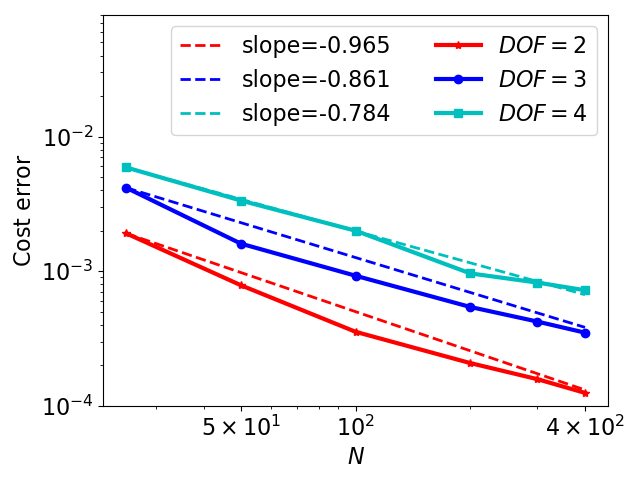}
    \caption{Numerical scaling results for the LQR operator-learning task. Left: RMSE versus the number of training samples N (log–log scale). Right: cost error versus $N$ (log–log scale). Dashed lines denote least-squares linear fits in log–log coordinates (matching the curve colors), and the reported slopes indicate the empirical decay rates.}
    \label{fig:scaling_lqr}
\end{figure}

\begin{figure}[htbp]
    \centering 
    \includegraphics[width=0.43\linewidth]{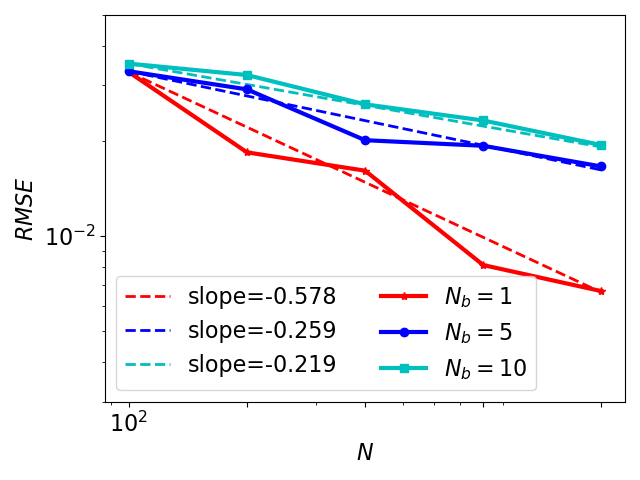}
    \includegraphics[width=0.43\linewidth]{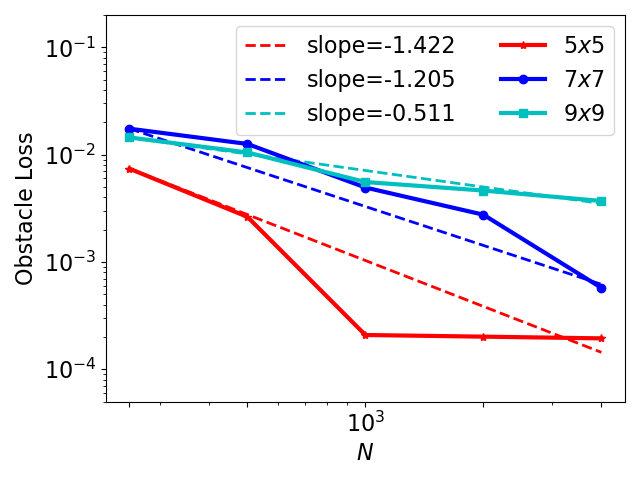}
    \caption{Numerical scaling results across two problems. Performance versus the number of training samples N (log–log scale) for (left) the 2D stationary Gaussian-bump problem with varying number of bumps $N_b$, and (right) the 2D maze navigation problem with different maze sizes. For each setting, dashed lines show least-squares linear fits in log–log coordinates (same color as the corresponding curve), and the reported slopes indicate the empirical decay rates. For the Gaussian-bump problem, we report the relative mean squared error between the neural-operator predictions and the reference solutions, where the references are obtained by solving \cref{eqn:Loss_gauss,eqn:dynamics_gauss} for each test instance. For the maze navigation problem, we report the obstacle loss defined in \cref{eqn:Loss_maze}, aggregated over mazes of each size.}
    \label{fig:scaling}
\end{figure}
%%%%%%%%%%%%%% scaling laws section new %%%%%%%%%%%%%%%%

\section{Related Works}
\textbf{Solution/Control learning.}
Early works on learning-based optimal control primarily aim to approximate the feedback controller for a single problem instance in a high-dimensional state space, using various function approximators such as neural networks. Different learning formulations have been explored, including direct policy optimization~\cite{han2016deep,bottcher2022ai,ainsworth2021faster}, backward stochastic differential equations-based methods~\cite{zhou2021actor,li2024neural}, physics-informed neural networks~\cite{mowlavi2023optimal,norambuena2024physics}, and supervised learning from open-loop optimal control data~\cite{nakamura2021adaptive,azmi2021optimal,hu2025learning,nakamura2022neural}. Several studies have further extended these approaches to multi-agent or mean-field control settings~\cite{han2020deep,ruthotto2020machine,xuan2022optimal,onken2021neural}.
While these models can achieve high accuracy for a fixed environment, they must be retrained whenever system parameters or cost functions change. This limitation motivates the use of neural operators, which amortize computation across problem families.

\textbf{Neural Operators.} 
Neural operators are designed to learn mappings between infinite-dimensional function spaces. Prominent examples include Fourier Neural Operators (FNO) \cite{li2020fourier,li2024physics,kovachki2023neural}, Deep Operator Networks (DeepONet) \cite{lu2021learning,wang2021learning}, and transformer-based architectures \cite{li2022transformer,hao2023gnot}. Compared with classical numerical solvers, well-trained neural operators offer substantial speedups while maintaining acceptable accuracy. A growing body of work leverages these advantages in optimal control. The existing application areas can be categorized as follows. %\jh{add mentioning to https://arxiv.org/abs/2402.10033 somewhere?}
%Motivated by these efficiency advantages, a growing body of work has applied neural operators to optimal control problems. In the following, we categorize and review existing approaches to neural operator–based optimal control problems.

\begin{itemize}

    \item \textbf{Learning Dynamics.} Neural operators serve as flexible surrogates for system dynamics or unknown nonlinearities. \cite{chi2024neural} jointly perform system identification and controller learning, while others~\cite{wang2021fast,lundqvist2025residual} learn a mapping  $\mathcal{K}_\theta:\bu(t) \mapsto \bx(t)$ that satisfy the dynamics \eqref{eqn:dynamics} approximately and then solve $\min_\bu J_B(\bu,\mathcal{K}_\theta(\bu))$. Extensions also handle nonsmooth objectives via primal–dual optimization~\cite{song2024operator}.

    \item \textbf{Enforcing Optimality Conditions.}
Other studies integrate neural operators into frameworks derived from optimality systems.
\cite{yong2024deep} learn mappings from initial states and environment parameters to optimal controls via first-order adjoint constraints.
\cite{lee2025hamilton} approximate mappings from terminal cost functions to value functions in the Hamilton–Jacobi–Bellman (HJB) equation, from which optimal feedback controls are obtained by minimizing the Hamiltonian using the gradient of this value function. \cite{verma2024neural} study parameterized optimal control problems by learning value functions or policies over finite-dimensional parameter spaces using HJB-based losses or reinforcement learning.  In motion-planning settings, where the HJB equation reduces to the Eikonal PDE, neural operators are trained to map environment-dependent cost functions to the corresponding value functions~\cite{matada2024generalizable}.

    \item \textbf{Controller Synthesis.} Unlike the optimal control formulation \eqref{eqn:Loss},\eqref{eqn:dynamics}, another line of research applies neural operators directly to controller synthesis, particularly within the backstepping framework for boundary control of PDEs. Here the primary goal is \textit{stabilization}, achieved by designing a feedback law expressed through an integral kernel (gain function) that satisfies an auxiliary kernel PDE. Computing this kernel is often costly, especially when system parameters vary. Neural operators have therefore been proposed to approximate the mapping from PDE coefficients to the integral kernel \cite{bhan2023neural,krstic2024neural,lamarque2025adaptive}.

    \item \textbf{Mean-Field Games (MFG).} \cite{huang2024unsupervised} propose an operator-learning framework mapping initial and terminal distributions to population trajectories, enabling fast inference for new MFG instances with a fixed environment. Our focus, by contrast, lies on task-varying environments; extending to the distributional MFG setting is a promising direction for future work.
\end{itemize}

\section{Conclusion and Broader Impact}
We propose a self-supervised neural operator learning method that amortizes optimal control by learning a direct map from problem conditions to optimal control functions, enabling near-instant inference after training and delivering strong performance across diverse benchmarks. Alongside these gains, we also characterize the method’s limitations, the `price to pay' for accuracy, both theoretically and empirically: our main scaling-law theorem quantifies how sample and model complexity grow with intrinsic dimension and ease with smoothness, and our experiments show corresponding performance degradation as problem complexity increases. Taken together, our results sharpen the scope of neural operators: they are neither a universal solver nor a narrow surrogate, but a powerful tool whose promise and limitations can now be clearly delineated.

The proposed method admits natural extensions to more complex optimal control settings. While this paper focuses on unconstrained optimal control problems, the approach can accommodate state and control constraints by augmenting the training objective with appropriate penalty, barrier, or augmented-Lagrangian terms to promote feasibility along trajectories. A further extension is to stochastic optimal control problems, where environment- or control-induced randomness enters the dynamics; handling this rigorously will require dedicated treatment, which we leave for future work.

% \wx{Another potential of our approach is its applicability to a closely related field: reinforcement learning (RL). The major distinction between optimal control and reinforcement learning lies in the assumptions on system knowledge—OCP typically presumes access to a reasonably accurate model and cost function, while RL is often deployed in a model-free setting with only observed rewards. In many real-world applications, however, partial knowledge of the underlying dynamics is available, meaning one need not adopt a fully model-free approach. While reinforcement learning policies can be evaluated cheaply after training, they do not generalize zero-shot to new environments. In contrast, our operator-learning framework is designed for true zero-shot generalization across tasks, as the learned operator directly maps environment functions to optimal controls. Moreover, unlike most RL approaches, our method comes with a theoretical analysis of sample complexity and generalization error, providing rigorous guarantees on efficiency and reliability in addition to strong empirical performance.} \jh{I think it is a good idea to add a discussion with RL. But the criticism of "While reinforcement learning policies can be evaluated cheaply after training, they do not generalize zero-shot to new environments. " might be too strong. RL is good at solving maze navigation type tasks, but via treating environment as part of the state variable. I make a more conservative point in Section 1.}

\section{Acknowledgment}
R. Lai's reserach is supported in part by NSF DMS-2401297. 

\appendix

\section{Notation Table}
We list all notations in Table~\ref{tab:notation}. 
\begin{table}[h]
    \centering
    \caption{Notation table}
    \label{tab:notation}
    \begin{tabular}{c|c}
    \hline
    Notation & Description  \\
    \hline
    \hline
    $\Omega$ and $D$ & Ambient space $\Omega = \mathbb{R}^D$ \\
    $\mX$ and $d$ & Initial state space, supported on a d-dimension manifold on initial state $\mX \subset \Omega$ \\
    $\bx_0 \in \mX$ & Initial state \\
    $T$ & Terminal time \\
    $\mM$ and $k$ &  barrier function space, supported on k-dimensional manifold $\mM \subset C(\Omega)$ \\
    $B \in \mM$ & barrier function \\
    $L(t, \bu, B)$ & Running cost function \\
    $M(\bx)$ & Terminal cost function \\
    $\bu^*(t; B, \bx_0)$ & Optimal control for \eqref{eqn:Loss} with initial state $\bx_0$ and barrier function $B$ \\
    $\bx^*(t)$ & Trajectory obtained by optimal control $\bu^*(t; B, \bx_0)$ \\
    $\mG^*(B, \bx_0)$ & Minimizer of the population loss coupled with the dynamics \eqref{eqn:amortize_poploss} \\
    $\mathscr{G}(R,\kappa,l,p,K)$ & A class of ReLU networks \\
    $\Gamma$ & Training set $\Gamma = \{(B^i,\bx_0^i)\}_{i=1}^n \in \mX\times \mM$ \\
    $\mG_\theta$ & Neural network approximation to operator $\mG$ \\
    $\bx_\theta(t)$ & Trajectory obtained by neural network $\mG_\theta(t, B, \bx_0)$ \\
    $\mG_{\hat{\theta}}$ & A ReLU network in class $\mathscr{G}(R,\kappa,l,p,K)$ \\
    % $\bx_\theta(t)$ & Trajectory produced by $\mG_\theta$  \\
    $\mR(\mG)$ & Population loss \\
    $\mR_{\Gamma}(\mG)$ & Empirical loss over dataset $\Gamma$ \\
    $\mG_{\Gamma;\theta^*}$ & Empirical risk minimizer over dataset $\Gamma$ within the class $\mathscr{G}(R,\kappa,l,p,K)$  \\
    \hline
    \hline
    $n$ & Number of training samples \\
    $m$ & Number of testing samples \\
    $N_t$ & Number of uniform time discretization steps \\
    $N_b$ & Number of Gaussian bumps \\
    $B(\bx)= \sum_{j=1}^{N_b} h_j \mN(\bx; \bc_j, \sigma_j I)$ & Stationary Gaussian mixture barrier \\
    $B(t, \bx) = \sum_{j=1}^{N_b} h_j \mN(\bx; \bc_j(t), \sigma_j I)$ & Moving Gaussian mixture barrier \\
    $B(\bx) = \sum_{j=1}^{N_b} B_j(\bx; a^j, b^j, c^j, d^j)$ & Maze configuration function \\
    \hline
    \hline
    $l_1$ & Number of layer of the barrier function encoder \\
    $l_2$ & Number of intermediate layer  \\
    $l$ & Number of total layers \\
    $n_r$ & Number of neurons \\
    $l_r$ & Learning rate \\
    $B_s$ & Batch size \\
    $D_s$ & Decay step: number of steps before the learning rate decays. \\
    $D_r$ & Decay rate \\
    $N_{epoch}$ & Number of training epochs\\
    \hline
    \end{tabular}
\end{table}

\begin{table}
    \centering
    \caption{Hyperparameters table}
    \label{tab:hyper}
    \begin{tabular}{|c|c|c|c|c|c|c|}
    \hline
    & \begin{tabular}{@{}c@{}}Gaussian Bumps \\ (2D / 3D) \\ (Stationary / Moving)\end{tabular}  & Maze & Unicycle  \\
    \hline
    \hline
    Time discretization $N_t$ & 30 & 500 & 30 \\
    Number of training sample $n$ & 1600 & 4000 & 2000  \\
    \hline
    \hline
    $l_1$ & 2  & 4  & 4  \\
    $l_2$ &  2  & 4  & 2 \\
    $n_r$ & 200 & 300 & 300 \\
    $n_h$ & - & 2 & - \\
    Fourier feature $\sigma'$ & - & $1.5\pi$ & - \\
    \hline
    \hline
    Leaning rate $l_r$ & 1e-4 & 5e-5 & 1e-4 \\
     Decay step $D_s$ & 500 & 200 & 400 \\
    Decay rate $D_r$ & 0.95 & 0.9 & 0.9 \\
    Batch size $B_s$ & 32 & 20 & 32\\
    Epoch $N_{epoch}$ & 4000 & 4000 & 4000 \\
    \hline
    \end{tabular}
\end{table}

\section{Supplementary Details for Maze Navigation Problem}\label{sec:network_maze}

The generation of the maze configurations follows the randomized Depth-First Search (DFS) algorithm. The Randomized Depth-First Search (DFS) algorithm generates a maze by exploring cells in a loop until all are visited. Starting from a random cell, it marks it visited, randomly chooses an unvisited neighbor, removes the wall between them, and moves forward while pushing the current cell onto a stack. If no unvisited neighbors remain, it backtracks by popping from the stack, repeating this process until the grid is fully explored. The result is a perfect maze with a single path between any two cells. In addition, we define a loopiness coefficient $p_\zeta \in [0,1]$ is applied after generation: with probability $p_\zeta$, extra walls are removed even between already-connected regions, introducing loops. See Algorithm~\ref{alg:maze} in Appendix~\ref{sec:network_maze} for details.
Maze generation is carried out by the following Algorithm~\ref{alg:maze}. Once the maze configuration is obtained, it can be converted into a list of axis-aligned rectangles within the domain $[0,1]^2$, which define the wall structures. Consequently, each maze is represented as a collection of rectangles, where each rectangle is uniquely specified by the tuple $[a,b,c,d]$. Here $a,b$ denotes the coordinates of the lower-left corner, and $c,d$ denotes the coordinates of the upper-right corner. Moreover, we fix start state at $\bx_0=(1,1)$ and the target terminal state at $\bx_T=(0,0)$.

\begin{algorithm}[htp]
\caption{Maze Generation with DFS and Loops}
\label{alg:maze}
\begin{algorithmic}[1]
\Require width $w$, height $h$, start $(s_x,s_y)$, end $(e_x,e_y)$, loop factor $\lambda$
\Ensure Binary grid \texttt{maze} of size $h \times w$
\State Ensure $w,h$ are odd (increment by 1 if even).
\State Initialize \texttt{maze} $\gets 0_{h \times w}$, \texttt{visited} $\gets \textbf{false}_{h \times w}$.
\State Snap start $(s_x,s_y)$ and end $(e_x,e_y)$ to even coordinates.
\Function{Visit}{$x,y$}
  \State \texttt{visited[$y,x$]} $\gets$ \textbf{true}; \texttt{maze[$y,x$]} $\gets 1$
  \State Let $\mathcal{D} = \{(2,0),(-2,0),(0,2),(0,-2)\}$ \Comment{possible moves}
  \State Randomly shuffle $\mathcal{D}$
  \For{each $(\Delta x, \Delta y) \in \mathcal{D}$}
    \State $n_x \gets x+\Delta x$, $n_y \gets y+\Delta y$
    \If{$0 \leq n_x < w$ and $0 \leq n_y < h$ and not \texttt{visited[$n_y,n_x$]}}
        \State \texttt{maze[$y+\Delta y/2,\,x+\Delta x/2$]} $\gets 1$ \Comment{carve wall}
        \State \Call{Visit}{$n_x,n_y$}
    \EndIf
  \EndFor
\EndFunction
\State Call \Call{Visit}{$s_x,s_y$}
\State $N \gets \lfloor \tfrac{w}{2}\rfloor \cdot \lfloor \tfrac{h}{2}\rfloor$ \Comment{cell count as in code}
\For{$i = 1$ to $\lfloor \lambda N \rfloor$}
  \Repeat
    \State Sample random interior coordinate $(x,y)$
  \Until{maze[$y,x$] = 0 and (vertical or horizontal neighbors are both open)}
  \State \texttt{maze[$y,x$]} $\gets 1$ \Comment{add loop}
\EndFor
\State Ensure start and end positions are open.
\State \Return \texttt{maze}, $(s_x,s_y)$, $(e_x,e_y)$
\end{algorithmic}
\end{algorithm}

\subsection{Network Architecture}
To encode complex maze configurations, we adopt a more advanced architecture than the deep ReLU networks used in the Gaussian-bump and unicycle problems. Our design combines Fourier feature embeddings~\cite{tancik2020fourier} with self-attention blocks~\cite{vaswani2017attention}. The maze configuration, represented as a list of axis-aligned rectangles, is first passed through a ReLU MLP encoder with $l_1$ layers and $n_r$ hidden neurons, followed by self-attention layers without positional encodings with $n_h$ heads, and then summarized using a $n_h$-head attention-pooling layer to obtain a permutation-invariant maze embedding. Temporal information is encoded separately using Fourier features, which provide high-frequency components and thus enable the model to capture sharp turns needed to avoid walls. Typically, we consider following encoding
\[
\gamma(t) = \begin{bmatrix}
\cos(\mathbf{B}t) \\
\sin(\mathbf{B}t)
\end{bmatrix},
\]
where each entry in $\mathbf{B}$ is sampled from $\mN(0, \sigma')$. Since the start and end states are fixed, they are not used as inputs. Finally, the temporal features and maze embedding are concatenated and passed through a deep ReLU MLP with $l_2$ layers and $n_r$ hidden neurons to produce the final prediction. The selection of the hyperparameters is summarized in Table~\ref{tab:hyper}.

\section{Implementation of Unicycle with Model Predictive Control}\label{sec:MPC_details}
We use $w_1=1$, $w_2=400$, $w_3=200$ and $w_4 = 2$ in \eqref{eqn:free_time_loss}. And we use two separate networks to predict the optimal control $\bu^*(t)$ and the terminal time $T^*$: ControlNet and TerminalTimeNet, respectively. ControlNet constructs a 12-dimensional obstacle token (center, width, height, and the initial and target states), embeds it into a 32-dimensional representation, and encodes time using 64-dimensional Fourier features projected into another 32-dimensional query token. These tokens are processed by a 2-layer Transformer encoder with 4 heads. The transformed query, combined with the mean-pooled obstacle features, is passed through a small MLP, and the outputs are scaled via tanh to produce bounded steering and acceleration controls.

TerminalTimeNet builds a 12-dimensional feature vector for each obstacle (center, width, height, initial state, target state) and feeds it into a light MLP. The network consists of 2 fully connected layers with 32 hidden units and ReLU activations, producing a scalar terminal-time contribution for each obstacle. These contributions are averaged across all obstacles, and a softplus is applied to ensure a strictly positive predicted terminal time.

During training, we fix the static-wall potential
\[
\sum_{j=1}^{N_{b1}} B_j(\bx; a^j, b^j, c^j, d^j),
\]
where we use three walls ($N_{b_1} = 3$) with parameters $[(0.1, 0.25, 0.4, 0.43), (0.5, 0.7, 0.67, 0.7), (0.6, 0.63, 0.2, 0.3)]$. 
We generate 4000 training data, in each training sample, we randomly sample up to two Gaussian bumps ($N_{b_2} \leq 2$): 
\[
\sum_{l=1}^{N_{b2}} \mN_l(\bx; \bc^l, \sigma^l)
\]
where each bump center is drawn from $\mathbf{c}_j \sim U([0.25,0.8]^2)$, the standard deviation is fixed at $\sigma_j = 0.02$, and configurations overlapping with walls are discarded. The system dynamics are discretized uniformly with $N_t = 80$.

\bibliographystyle{plain}  % or use IEEEtran, alpha, etc.
\bibliography{references}  

\end{document}